\documentclass[11pt]{article}

\usepackage{latexsym}
\usepackage{amssymb}
\usepackage{amsthm}
\usepackage{amscd}
\usepackage{amsmath}
\usepackage{verbatim}
\usepackage{xcolor}
\usepackage{enumerate}
\usepackage{graphicx}

\newtheorem{theorem}{Theorem}[section]

\newtheorem{lemma}[theorem]{Lemma}
\newtheorem*{lemma*}{Lemma}

\newtheorem*{definition*}{Definition}
\newtheorem{definition}[theorem]{Definition}
\newtheorem{proposition}[theorem]{Proposition}
\newtheorem{corollary}[theorem]{Corollary}

\theoremstyle{definition}

\theoremstyle{remark}
\newtheorem{remark}{Remark}[section]

\numberwithin{equation}{section}

\let\inf\relax \DeclareMathOperator*\inf{\vphantom{p}inf}

\newcommand{\leqn}{\begin{equation}\label}
\def\endeqn{\end{equation}}

      \def\RR{\mathbb{R}}

\newcommand{\NN}{\mathbb{N}}

\newcommand{\res}{\hbox{{\vrule height .22cm}{\leaders\hrule\hskip .2cm}}}

\makeatletter
\renewcommand{\@makefnmark}{\mbox{\textsuperscript{}}}
\makeatother

\def\adots{\mathinner{\mkern2mu\raise0pt\hbox{.}  
\mkern2mu\raise4pt\hbox{.}\mkern1mu
\raise7pt\vbox{\kern7pt\hbox{.}}\mkern1mu}}
\def\res{\hbox{ {\vrule height .22cm}{\leaders\hrule\hskip.2cm} } }

\allowdisplaybreaks[1]

\begin{document}

\title{A sharp bound on the Hausdorff dimension of the singular set of a uniform measure}
\author{ A. Dali Nimer}
\date{}
\maketitle\footnote{The author was partially supported by NSF RTG 0838212, DMS-1361823 and DMS-0856687}
\footnote{Department of Mathematics, University of Washington, Box 354350, Seattle, WA
98195-435.

E-mail address: nimer@uw.edu}

\begin{abstract}
The study of the geometry of $n$-uniform measures in $\RR^{d}$ has been an important question in many fields of analysis since Preiss' seminal proof of the rectifiability of measures with positive and finite density.
The classification of uniform measures remains an open question to this day. In fact there is only one known example of a non-trivial uniform measure, namely $3$-Hausdorff measure restricted to the Kowalski-Preiss cone.
Using this cone one can construct an $n$-uniform measure whose singular set has Hausdorff dimension $n-3$. 
 
 In this paper, we prove that this is the largest the singular set can be. Namely, the Hausdorff dimension of the singular set of any $n$-uniform measure is at most $n-3$.
\end{abstract}

\section{Introduction}

We study the geometry of the singular set of $n$-uniform measures. Understanding the geometry of $n$-uniform measures has been an important question in geometric measure theory ever since they first appeared in the proof of Preiss' theorem on the rectifiability of measures. 
Indeed, to describe the local and global structure of a measure, Preiss defined its tangent measures:  those roughly consist in blowing up (zooming in) or blowing down (zooming out) a measure at a point. It turns out that for a measure having positive finite density, almost all of its tangents are $n$-uniform.
 Since the tangents of any measure with `nice' density ratio are $n$-uniform, the study of the geometry of uniform measures is crucial in many contexts from geometric measure theory and harmonic analysis to PDE's (see $\cite{DKT}$, $\cite{PTT}$ or $\cite{KT}$ for examples). More precisely, any dimension reduction argument for the singular set of $n$-uniform measures should lead to an improvement in the description of the singular set of measures whose tangents are $n$-uniform.
 
 The investigation in this paper is motivated by the open question of the classification of $n$-uniform measures. Indeed, the only known $n$-uniform measures are the flat ones (namely $n$-Hausdorff measure restricted to an $n$-plane), the Kowalski-Preiss cone (or KP cone) $C$ of dimension $3$ defined as
\begin{equation}
\label{KPcone}
C =\{(x_1,x_2,x_3,x_4); x_4^2=x_1^2+x_2^2+x_3^2 \}.
\end{equation}
and products of the former.
This paper gives a sharp bound on the Hausdorff dimension of the singular set of an $n$-uniform measure in $\RR^{d}$, where a point is called singular if its tangent is not $n$-flat: we prove that the Hausdorff dimension of the singular set of an $n$-uniform measure is at most $n-3$. 
This bound effectively proves that the case of the KP-cone is in fact the worst in terms of the dimension of its singular set.

A measure $\Phi$ is $n$-rectifiable if it is absolutely continuous to $\mathcal{H}^n$ and there exists a countable collection of $C^1$ $n$-manifolds $\left\lbrace M_j \right\rbrace_{j}$ such that $\Phi (\RR^{d} \backslash \bigcup_{j} M_{j})=0$. In $\cite{P}$, Preiss proved the following remarkable theorem relating the rectifiability of a measure to its density.
\begin{theorem}[\cite{P}] \label{Preiss}
A Radon measure $\Phi$ of $\RR^d$ is $n$-rectifiable if and only if it satisfies the following property:
 \begin{equation}\label{densityproper} \mbox{ The density }
\Theta^{n}(x)= \lim_{r \to 0} \frac{\Phi(B(x,r))}{\omega_{n} r^{n}}
\mbox{ exists and is positive and finite for }\Phi-\mbox{a.e. } x \in \RR^{d}.
 \end{equation}
\end{theorem}
To prove this theorem, Preiss studies the geometry of $n$-uniform measures which appear as tangents (blow-ups) to measures satisfying $\eqref{densityproper}$.
A measure $\mu$ is said to be $n$-uniform if there exists a constant $c>0$ such that for any $x$ in the support of $\mu$ and any radius $r>0$, we have:
\begin{equation}
\mu(B(x,r))=cr^{n}.
\end{equation}
In $\cite{P}$, Preiss also provides a classification of the cases $n=1,2$ in $\RR^{d}$ for any $d$. In these cases, $\mu$ is $n$-Hausdorff measure restricted to a line or a plane respectively.

Interestingly, flat measures are not the only examples of uniform measures. Indeed, in $\cite{KoP}$, Kowalski and Preiss proved that $\mu$ is $(d-1)$-uniform in $\mathbb{R}^d$ if and only if $\mu=\mathcal{H}^{d-1} \res V$
where $V$ is a $(d-1)$-plane, or $d \geq 4$ and there exists an orthonormal system of coordinates
in which $\mu= \mathcal{H}^{d-1} \res (C \times W)$ where $W$ is a $(d-4)$-plane and $C$ is the KP-cone.
The classification for $n \geq 3$ and codimension $d-n \geq 2$ remains an open question.

Kirchheim and Preiss later proved in $\cite{KiP}$ that the support $\Sigma$ of a uniformly distributed measure (of which $n$-uniform measures are an example) is a real analytic variety, namely the intersection of countably many zero sets of analytic functions. An application of the stratification theorem for real analytic varieties implies that $\Sigma$ must be a countable union of real analytic manifolds and the singular set has Hausdorff dimension at most $(n-1)$.

We investigate the Hausdorff dimension of the singular set $\mathcal{S}_{\mu}$ of an $n$-uniform measure $\mu$. Our main result is the following theorem.

\begin{theorem} \label{maintheorem} Let $\mu$ be an $n$-uniform measure in $\RR^{d}$, $ 3 \leq n \leq d$.
Then
$$dim_{\mathcal{H}}(\mathcal{S}_{\mu}) \leq n-3.$$
\end{theorem}
In the cases $n=3$, $d=3$, it is a standard result that the only $3$-uniform measure (up to normalization) is $3$-Lebesgue measure. In this case, the bound is obvious. 
To see that this bound cannot be improved, let $n\geq 3$, $d>n$ and consider the measure $\mu$ defined as:
\begin{equation}
\mu = \mathcal{H}^{n} \res M,
\end{equation}
where $M$ is the set
\begin{equation}
M=\left\lbrace (x_1, \ldots, x_d) ; x_4^2=x_1^2+x_2^2+x_3^2 \mbox{ and } x_{n+1}=\ldots = x_d =0 \right\rbrace.
\end{equation}
By $\cite{KoP}$, since products of uniform measures are uniform, $\mu$ is $n$-uniform. Moreover, the singular set of $\mu$ is $\RR^{n-3}$ which has Hausdorff dimension $n-3$.

This theorem is first proven for the base case $n=3$. The crux of this argument is a theorem stating that conical $3$-uniform measures have at most one singularity at the origin. We then prove that singular sets of $n$-uniform measure behave nicely under blow-ups. This allows us to adapt Federer's dimension reduction argument to generalize our base case to any dimension $n$.

We briefly discuss the steps of our proof. In the first section, we limit our investigation to $n$-uniform conical measures. A conical uniform measure is a measure that is dilation invariant up to appropriate scaling. We first obtain a polar decomposition for such a measure. By polar decomposition, we mean that the measure decomposes into a spherical and a radial component. Using this decomposition, we isolate the spherical component of a $3$-uniform measure and prove that it is locally $2$-uniform. This allows us to deduce that every point of the spherical component is flat from which the following theorem follows.

 \begin{theorem}\label{theorem1} Let $\nu$ be a conical $3$-uniform measure in $\mathbb{R}^d$ and let $\Sigma$ be its support. Then there exists $\gamma>0$ such that $\Sigma \backslash \left\lbrace 0 \right\rbrace$ is a $C^{1,\gamma}$ submanifold of dimension $3$ . \end{theorem}
 
 In the case where $\nu$ is a conical $n$-uniform measure, for general $n$, the spherical component turns out to be uniformly distributed.
 This means that there exists a function $\phi: \RR_{+} \rightarrow \RR_{+}$ such that for any point $x$ of its support, and any positive radius $r>0$ $$\nu(B(x,r))=\phi(r).$$
 We use this result to show that for a conical $n$-uniform measure, Kirchheim and Preiss' result can be improved to an algebraic variety that is symmetric with respect to the origin.
 
 \begin{theorem}\label{theorem2}
  Let $\nu$ be a conical $n$-uniform measure in $\mathbb{R}^d$ and $\Sigma$ its support. Then $\Sigma$ is an algebraic variety and \begin{equation}
 \Sigma=-\Sigma.
 \end{equation}
 \end{theorem}

 In the second section, we first start by proving a lemma about the connectedness of  blow-ups along a sequence of points.
 This connectedness is expressed in terms of a measure's distance from flat measures.
 To this effect, we use a positive functional  $F$ defined on Radon measures satisfying $F(\mu)=0$ if and only if $\mu$ is flat .
 \begin{lemma}
 There exists $\epsilon_{0}>0$ such that the following holds. Let $\mu$ be an $n$-uniform measure in $\RR^{d}$, $\left\lbrace x_k \right\rbrace_{k} \subset \mbox{supp}(\mu)$ and $\left\lbrace \tau_{k} \right\rbrace_k$, $\left\lbrace \sigma_{k} \right\rbrace_{k}$ sequences of positive numbers decreasing to $0$.
We also assume that $\sigma_k < \tau_k$ and that there exist $n$-uniform measures $\alpha$ and $\beta$ such that:
$$\mu_{x_k,\tau_k} \rightharpoonup \alpha \mbox{ and } \mu_{x_k,\sigma_k} \rightharpoonup \beta.$$
Then:
$$F(\alpha) < \epsilon_{0} \implies F(\beta) < \epsilon_{0}.$$

 \end{lemma}
 
 We use this lemma to deduce a theorem about the convergence of singular sets. Roughly speaking, blow-ups preserve singularity.
 
 \begin{theorem} \label{theorem3}
 Let $\mu$ be an $n$-uniform measure in $\RR^d$, $x_0 \in \mbox{supp}(\mu)$, $\left\lbrace x_j \right\rbrace _{j} \subset \mathcal{S}_{\mu}$, $\left\lbrace r_j \right\rbrace _{j}$ any sequence of positive numbers decreasing to $0$.
Also assume that $y_j=\frac{x_j - x_0}{r_j} \in \overline{B(0,1)}$, $y_j \to y$.
Then $$y \in \mathcal{S}_{\nu},$$
where $\nu$ is the  tangent to $\mu$ at $x_0$ with appropriate normalization.
\end{theorem}

\section{Preliminaries}
Let us start by defining Hausdorff measure and the concepts of upper and lower density. Though these definitions are standard, we make them to keep track of the constants, especially in the second section.

\begin{definition}
We define $\omega_s$ to be the constant:
\begin{equation}\nonumber
\omega_s=\frac{\pi^{\frac{s}{2}}}{\Gamma(\frac{s}{2} +1)},
\end{equation}
so that in particular $\omega_m$ is the  volume of the unit ball $B^{m}(0,1)$ when $m \in \NN$.
For $\delta \in (0,\infty]$, define $\mathcal{H}^{s}_{\delta}$, $s \leq d$, to be the measure in $\RR^{d}$ defined in the following way. If $A \subset \RR^{d}$

\begin{equation}\nonumber
\mathcal{H}^{s}_{\delta}(A)=\omega_{s} \inf \sum_{j}^{\infty} \left( \frac{diam(E_j)}{2}\right)^{s}, 
\end{equation}
where the infimum is taken over all countable coverings $\left\lbrace E_j \right\rbrace_{j}$ of
 $A$ such that $diam(E_j)<\delta$.

Then define $s$-Hausdorff measure $\mathcal{H}^s$ to be 
\begin{align}
\mathcal{H}^s(A)&=\lim_{\delta \to 0}\mathcal{H}^{s}_{\delta}(A), \nonumber \\
&= \sup_{\delta>0}\mathcal{H}^{s}_{\delta}(A).
\end{align}
It is a standard result of measure theory that $\mathcal{H}^s$ is a Borel regular measure on $\RR^{d}$.
\end{definition}

\begin{definition}
Let $\Phi$ be a measure on $\RR^{d}$, $x \in \RR^{d}$. We define the lower density $\theta_{*}^{s}(\Phi, x)$ and upper density $\theta^{*, s}(\Phi, x)$ of $\Phi$ at $x$ to be 

\begin{align}
& \theta_{*}^{s}(\Phi, x) = \liminf_{r \to 0} \frac{\Phi(B(x,r))}{\omega_s r^{s}} \nonumber\\
& \theta^{*,s}(\Phi, x) = \limsup_{r \to 0} \frac{\Phi(B(x,r))}{\omega_s r^{s}}.
\end{align}
If the limsup and the liminf coincide, we call their common value the density of $\Phi$ at $x$ and denote it by $\theta^{s}(\Phi,x)$.

\end{definition}

\subsection{The area and co-area formula}
We will need the two following theorems in Section 3 of the paper. The co-area formula will allow us to decompose a conical uniform measure into a spherical and a radial component. 
As for the area formula, it will be used to compute the measure of a ball by the spherical component.

\begin{theorem}[{\cite{S}}][The area formula]
Let $f: \RR^{m} \rightarrow \RR^{d}$ be a 1-1 $C^{1}$ function where $m<d$. Then, for any  Borel set $A \subset \RR^{m} $,we have:
\begin{equation}\label{area}
\int_{A} Jf(x) d\mathcal{L}^{m}(x)= \mathcal{H}^{m}(f(A))
\end{equation}
where \begin{equation}
Jf(x)= \sqrt{det((df(x))^{*}  \circ df(x)) },
\end{equation}
and $(df(x))^{*}$ is the adjoint of $df(x)$.
\end{theorem}

The co-area formula can be viewed as a more general form of Fubini's theorem. To state it, we first need to define a notion of gradients for Lipschitz functions whose domain is a rectifiable set.

Let $M$ be an $n$-rectifiable set in $\RR^{d}$ (in particular $n \leq d$) that is, $M$ can be written as a countable union of $C^{1}$ manifolds $\left\lbrace N_i \right\rbrace $ up to a set of $\mathcal{H}^n$-measure zero. 
Let  $f: M \rightarrow \RR^{m}$ with $f=(f_1, \ldots, f_m)$ be a Lipschitz function. Then by Rademacher's theorem, $f$ is almost everywhere differentiable (the same holds for each $f_l$).
With this in mind, for $x \in N_j$ for some $j$ (in particular this is true for $\mathcal{H}^{n}$ almost every $x \in M$), $T_{x}M = T_{x}N_{j}$ the tangent plane at $x$ as a point of $N_j$. At $\mathcal{H}^n$-almost every point of $M$, we can define the gradient $\nabla^{M} f_{l} = \nabla^{N_j}f_{l}$ of $f_{l}$  and the linear map $d^{M}f(x): T_{x}M \rightarrow \RR^{m}$ in the following way:
\begin{equation}
 d^{M}f(x)(\tau)= \sum_{j=1}^{m} \left\langle \tau , \nabla^{M}f_{j}(x) \right\rangle e_{j},
\end{equation}
where $\left\lbrace e_j \right\rbrace$ is an orthonormal basis of $\RR^{m}$.

\begin{theorem}[\cite{S}][The co-area formula]
Let $M \subset \RR^d$ be an $n$-rectifiable set and $f: M \rightarrow \RR^{m}$, $m < n \leq d$  a Lipschitz function. Then for any non-negative Borel function $g:M \rightarrow \RR$, we have:
\begin{equation}\label{coarea}
\int_{M} g(x) J_{M}^{*}f (x)d\mathcal{H}^{n} (x) = \int_{\RR^{m}} \int_{f^{-1}(y) \cap M} g(z) d\mathcal{H}^{n-m}(z) d\mathcal{L}^{m}(y),
\end{equation}
where 
\begin{equation}
 J_{M}^{*}f(x)= \sqrt{det(d^{M}f(x) \circ (d^{M}f(x))^{*} ) }.
\end{equation}
\end{theorem}
\subsection{Weak convergence of measures and metrization of the space of Radon measures}
When studying the convergence of Radon measures, it is often very useful to metrize the space of Radon measures. We start by defining the notion of the support of a measure.
\begin{definition}
Let $\mu$ be a measure in $\RR^{d}$. We define the support of $\mu$ to be 
\begin{equation}
\mbox{supp}(\mu)=\left\lbrace x \in \RR^{d} ; \mu(B(x,r))>0, \mbox{ for all } r>0 \right\rbrace. 
\end{equation}
Note that the support of a measure is a closed subset of $\RR^{d}$.
\end{definition}
We can define weak convergence for a sequence of Radon measures.
\begin{definition}
Let $\Phi$, $\Phi_j$, $j>0$ be Radon measures in $\RR^{d}$.
We say that $\Phi_j$ converges weakly to $\Phi$ if for every $f \in C_{c}(\RR^{d})$, the following holds:
\begin{equation}
\int f(z) d\Phi_j (z) \to \int f(z) d\Phi(z).
\end{equation} 
We denote it by $ \Phi_j \rightharpoonup \Phi$.
\end{definition}
The results in this section appear in this form in $\cite{M}$.
\begin{theorem}
Let $\Phi_j$ be a sequence of Radon measures on $\RR^{d}$.Then $\Phi_j \rightharpoonup \Phi$, if and only if for any $K$ compact subset of $\RR^{d}$ and any $G$ open subset of $\RR^{d}$ the following hold:
\begin{enumerate}
\item $\Phi(K) \geq \limsup \Phi_{j}(K).$
\item $\Phi(G) \leq \liminf \Phi_{j}(G).$
\end{enumerate}

\end{theorem}

\begin{theorem}\label{weakconv}
Let $\Phi_j$ be a sequence of Radon measures on $\RR^{d}$ such that
$$ \sup_{j}(\Phi_j(K)) < \infty ,$$
for all compact sets $K \subset \RR^{d}$. Then there is a weakly convergent subsequence of $\Phi_j$.
\end{theorem}

We now want to define a metric on the space of Radon measures.

\begin{definition}\label{L(r)}
Let $0<r<\infty$. We denote by $\mathcal{L}(r)$ the set of all non-negative Lipschitz functions $f$ on $\RR^{d}$ with $spt(f) \subset B(r)$ and with $Lip(f) \leq 1$. For Radon measures $\Phi$ and $\Psi$ on $\RR^d$, set
$$F_{r}(\Phi, \Psi)= \sup\left\lbrace \left| \int f d\Phi - \int f d\Psi \right| : f \in \mathcal{L}(r) \right\rbrace.$$
We also define $\mathcal{F}$ to be 
$$\mathcal{F}(\Phi, \Psi) = \sum_{k} 2^{-k} F_{k}(\Phi,\Psi). $$ 
It is easily seen that $F_r$ satisfies the triangle inequality for each $r>0$ and that $\mathcal{F}$ is a metric.
\end{definition}

\begin{proposition}
Let $\Phi$, $\Phi_{k}$ be Radon measures on $\RR^{d}$. Then the following are equivalent:
\begin{enumerate}
\item $\Phi_j \rightharpoonup \Phi.$
\item $\lim \mathcal{F}(\Phi_j, \Phi) \to 0$
\item For all $r>0$, $\lim_{j \to \infty} F_{r}(\Phi_{j}, \Phi)=0.$

\end{enumerate}

\end{proposition}

\subsection{Tangent Measures and Uniform measures}
Let $\mu$ be a Radon measure on $\RR^{d}$ and $\Sigma$ its support.
For $a \in \RR^{d}$, $r>0$, define $T_{a,r}$ to be the following homothety that blows up $B(a,r)$  to $B(0,1)$:
\begin{equation} 
T_{a,r}(x)=\frac{x-a}{r} . \nonumber
\end{equation}
We define the image $T_{a,r}[\mu] $ of $\mu$ under $T_{a,r}$ to be the following measure:
\begin{align*}
T_{a,r}[\mu](A) & =\mu(T_{a,r}^{-1}(A)),\\ &=\mu(rA+a), \text{  } A \subset {\RR}^{d}. 
\end{align*}

\begin{definition}[\cite{P}]\label{tangent*}
We say that $\nu$ is a tangent measure of $\mu$ at a point $x_0 \in \RR^d$ if $\nu$ is a non-zero Radon measure on $\RR^n$ and if there exist sequences $(r_i)$ and $(c_i)$ of positive numbers such that $r_i \downarrow 0$ and:
\begin{equation} \label{tangent'}
c_i T_{x_0,r_i}[\mu]\rightharpoonup \nu \text{ as } i \rightarrow \infty, 
\end{equation}
where the convergence in ($\ref{tangent'}$) is the weak convergence of measures.
We write $\nu \in \mbox{Tan}(\mu,x_0)$.
 \end{definition}

\begin{remark}
 By Remark 14.4.3 in $\cite{M}$, if 
 \begin{equation} \label{upperlowerdensity}
 0< \Theta_{*}^{n}(\mu,x_0) \leq \Theta^{* n}(\mu,x_0) <\infty,
 \end{equation}
 and if $\nu \in \text{Tan}(\mu,x_0)$, then we can choose ($r_i$) such that:
 \begin{equation} \label{tangent}
 {r_{i}}^{-n} T_{x_0,r_i}[\mu]\rightharpoonup c \nu \text{ as } i \rightarrow \infty,
 \end{equation}
for some $c>0$. In the setting of this paper, $\eqref{upperlowerdensity}$ will always hold and we will only use $\eqref{tangent}$ when talking about tangent measures.
\end{remark}
\begin{definition}\label{flat}
A measure  on $\RR^d$ is called $n$-flat if it is equal to $ c \mathcal{H}^{n} \res V$, where $V$ is an $n$-plane, and $0<c<\infty$.

Let $\mu$ be a Radon measure on $\RR^d$ and $x_0$  be a point in the support $\Sigma$ of $\mu$. We will call $x_0$ a flat (or regular) point of $\Sigma$ if there exists an $n$-plane $V$ such that
\begin{equation}\label{flatpoint}
\text{Tan}(\mu,x_0)= \left\lbrace c \mathcal{H}^{n} \res V \ ; \ c >0 \right\rbrace. 
\end{equation}

Any point of $\Sigma$ that is not flat will be called a singular (or non-flat) point.

\end{definition}

\begin{definition}
Let $\mu$ be a Radon measure in $\RR^d$.
\begin{itemize}
\item We say $\mu$ is uniformly distributed if there exists a positive function $\phi : \RR_{+} \rightarrow \RR_{+}$ such that:
$$
\mu(B(x,r))=\phi(r), \text{ for all } x \in \Sigma, r>0.$$
We call $\phi$ the distribution function of $\mu$.

\item If there exists $c>0$ such that $\phi(r)=c r^n$, we say $\mu$ is $n$-uniform.

\item If $\mu$ is an $n$-uniform measure such that $T_{0,r}[\mu] = r^{n} \mu $ for all $r>0$, we call it a conical $n$-uniform measure.

\end{itemize}

\end{definition}

In [$\cite{P}$, Theorem $3.11$], Preiss showed that if $\mu$ is an $n$-uniform measure, there exists a unique $n$-uniform measure $\lambda$ such that:
\begin{equation}
r^{-n} T_{x,r} [\mu] \rightharpoonup \lambda, \mbox{ as } r \to \infty, 
\end{equation}
for all $x \in \RR^{d}$. $\lambda$ is called the tangent measure of $\mu$ at $\infty$.

The following theorem describes a basic but essential property of uniformly distributed measures: how radial functions integrate against them.

\begin{theorem}\label{radial}
Let $\mu$ be a uniformly distributed measure on $\RR^d$ and $f$ be a  non-negative Borel function on $\RR_{+}$. For all $z , y \in \mbox{supp}(\mu)$, we have:
$$ \int f(|x-z|) d\mu(x) = \int f(|x-y|)d\mu(x).$$
\end{theorem}
\begin{proof}
This is a simple application of Fubini's theorem. Indeed, if $f= \alpha \chi_{I}$, where $\alpha\geq 0$ and $I=(c,d)$ is an interval
\begin{align*}
\int f(|x-z|) d\mu(x) &= \alpha \int_{0}^{1} \mu(\left\lbrace x; \chi_{I}(|x-z|) \geq t \right\rbrace) dt, 
                            \\&=\alpha \left( \mu(B(z,d) \cap {B(z,d)}^{c})\right) , 
                            \\&=\alpha \left( \mu(B(y,d)\cap {B(y,c)}^{c}) \right) , \mbox{ since } \mu(B(z,r)) = \mu(B(y,r)) \mbox{ for all }r
                            \\&=\int f(|x-y|) d\mu(x).
\end{align*}
The result follows for general non-negative Borel functions by linearity of the integral and density of step functions.
\end{proof}
In $\cite{P}$, Preiss introduced the following $k$-forms which were essential to understand the structure of uniform measures.

\begin{definition}[3.4.(1), \cite{P}]
For $\mu$ uniformly distributed measure in $\RR^d$, $s>0$ and $k \in \NN$, define the following symmetric $k$-linear form $b_{k,s} \in {\odot}^{k} {\RR}^d$:

\begin{equation} 
b_{k,s}(u_1 \odot \ldots \odot u_k)= (2s)^k {(I(s) k!)}^{-1} \int \langle z,u_1 \rangle \ldots \langle z,u_k \rangle e^{-s|z|^{2}} d\mu(z), 
\end{equation}
where $$ I(s)= \int e^{-s|z|^2} d\mu(z). $$
\end{definition}

We will quote a theorem by Preiss describing Taylor expansions of those forms, and two consequences of this expansion.

\begin{theorem}[3.6, \cite{P}] \label{moments}
Let $\mu$ be a uniformly distributed measure in $\RR^d$. 
\begin{enumerate}
\item There are symmetric forms $b_{k}^{(j)} \in {\odot} ^{k} \RR^d$ such that:
  \begin{enumerate}
  \item $\label{1a}$ $b_{k,s}= \sum_{j=1}^{q} s^{j} \frac{b_{k}^{(j)}}{j!} + o(s^q)$ as $s \downarrow 0$ for every $k=1,2,\ldots$ and every $q=1,2,\ldots$. 
  \item $b_{k}^{(i)}=0$ whenever $2i<k$.
  \item$\sum_{k=1}^{2q} b_{k}^{(q)}(x^k)= {|x|}^{2q}$ for every $q=1,2,\ldots$ and every $x \in \Sigma$.
  \end{enumerate}
  Moreover, the forms $b_{k}^{(j)}$ are uniquely determined by ($\ref{1a}$).
  \item There are symmetric forms ${{\hat{b}}_{k}^{(j)} \in \odot} ^{k} \RR^d$ such that:
  \begin{enumerate}
  \item $\label{2a}$ $s^{-k} b_{k,s}=\sum_{j=1}^{q} s^{-j} \frac{{\hat{b}}_{k}^{(j)}}{j!} + o(s^{-q})$ as $s \uparrow \infty$, for every $k=1,2,\ldots$ and every $q=1,2,\ldots$, and
  \item ${\hat{b}}_{k}^{(i)}=0$ whenever $k>2i$.
  \end{enumerate}
  Moreover, the forms ${\hat{b}}_{k}^{(j)}$ are determined uniquely by ($\ref{2a}$).
  \end{enumerate}
\end{theorem}
 
 If $\mu $ is assumed to be conical, one gets the following improvement on Theorem $\ref{radial}$ and Theorem $\ref{moments}$. 

\begin{theorem}[3.10,\cite{P}] \label{momentsconical}
Let $\mu$ be a uniformly distributed conical measure. Then there exists $n$ such that $\mu$ is $n$-uniform and:
\begin{itemize}
\item if $x \in \Sigma$ and $\lambda>0$, then $\lambda x \in \Sigma$, where $\Sigma$ is the support of $\mu$.
\item whenever $u \in \Sigma$, $e \in {\RR}^{n}$, $|u|=|e|$ and $f$ is a non-negative Borel function on ${\RR}^2$ then:
\begin{equation} \label{protopolar}
\int_{{\RR}^d} f({|z|^{2}}, {\langle z,u \rangle}) d\mu(z)= C \int_{{\RR}^{n}}f({|x|^2},{\langle x,e \rangle})d\mathcal{L}^{n}(x).
\end{equation}
\item For every $s>0$ and $k=1,2,\ldots$, we have
\begin{equation} \label{conicalevenodd} b_{2k-1,s}=0 \mbox{ and } b_{2k,s}=\frac{s^{k}}{k!} b_{2k}^{k}. \end{equation}
\item If $\Sigma$ denotes the support of $\mu$ we have: \begin{equation}\label{varietyconical} \Sigma \subset \bigcap_{k>0} \left\lbrace x \in {\RR}^d ;  b_{2k}^{k} (x^{2k})=|x|^{2k} \right\rbrace. 
\end{equation}
\end{itemize}
\end{theorem}

 The following theorem is an important consequence of Theorem $\ref{moments}$. Note that the statement here is slightly different from Theorem 3.11 in $\cite{P}$. Indeed, we restate this theorem on convergence of measures in term of the metric from Definition $\ref{L(r)}$.
 
 \begin{theorem}[3.11, \cite{P}]\label{uniqueness}
 Let $\mu$ be an $n$-uniform measure in $\RR^{d}$. Then, for every $x \in \Sigma \cup \left\lbrace \infty \right\rbrace$, there exists a unique conical $n$-uniform measure ${\lambda}_x$ such that:
 \begin{itemize}
 \item Tan($\mu$,$x$)=$\left\lbrace c \lambda_{x} ; c>0 \right\rbrace$
 \item $\lim_{r \to 0} \mathcal{F}(r^{-n} T_{x,r}[\mu] , {\lambda}_{x})= 0$ if $x \neq \infty$.
 \item $\lim_{r \to \infty} \mathcal{F}(r^{-n} T_{y , r} [\mu] , \lambda_{\infty})=0$ for each $y \in \RR^{d}$.
\end{itemize}
 Moreover, for $\mu$-almost every $x \in \Sigma$, ${\lambda}_{x}$ is flat.

\end{theorem}  

We know from Theorem $\ref{Preiss}$ that an $n$-uniform measure is $n$-rectifiable. We can translate this into a corollary on the rectifiability of the support of an $n$-uniform measure.

\begin{corollary}\label{supportrect}
Let $\mu$ be an $n$-uniform measure in $\RR^{d}$ with $\Sigma=\mbox{supp}(\mu)$ and let $c>0$ be such that for $x \in \Sigma$, $r>0$ \begin{equation}
\mu(B(x,r))=cr^{n}.
\end{equation} 
Then  $\Sigma$ is $n$-rectifiable and 
\begin{equation}\label{supportunifmeasure}
\mu = c  \omega_{n}^{-1}\mathcal{H}^{n} \res \Sigma.
\end{equation}
\end{corollary}
\begin{proof}
By Theorem $\ref{PTT2}$, since $\mu$ is $n$-uniform, $\Sigma$ is a $C^{1,\alpha}$ $n$-manifold in the neighborhood of $\mathcal{H}^{n}$-almost every point. In particular, denoting the $n$-density of $\Sigma$ at $x$ by $\theta^{n}(\Sigma,x)$, we  have:
\begin{equation}\label{densitymanifoldish}
\theta^{n}(\Sigma,x)=1, \mbox{ for } \mathcal{H}^{n} \mbox{ almost every } x \in \Sigma.
\end{equation}
Let $D(x)$ denote $D(\mu, \mathcal{H}^{n}, x)$ the Radon-Nikodym derivative of $\mu$ with respect to $\mathcal{H}^n$ at $x$.
For $x \in \Sigma$, 
\begin{align*}
D(x) &=\theta^{n}(\mu,x) \theta^{n}(\Sigma,x)^{-1}, \\
        &=c \omega_{n}^{-1}\theta^{n}(\Sigma,x)^{-1}.
  \end{align*}
Theorem $2.12$ from $\cite{M}$ implies that for all $A \subset \Sigma$ 
\begin{equation}\label{abscont}
\mu(A)=c \omega_{n}^{-1} \int_{\Sigma} \theta^{n}(\Sigma,x)^{-1} d\mathcal{H}^{n}(x).
\end{equation}
Combining $\eqref{densitymanifoldish}$ and $\eqref{abscont}$, we get
 \begin{equation}\label{rect1}
\mu = c \omega_{n}^{-1} \mathcal{H}^{n} \res \Sigma.
\end{equation}

Now since $\mu$ is $n$-rectifiable by Theorem $\ref{Preiss}$, there exists an $n$-rectifiable set $M$ such that:
\begin{equation}\label{rect2}
\mu(\RR^{d} \cap M^{c})=0.
\end{equation}
Combining $\eqref{rect1}$ and $\eqref{rect2}$, we see that there exists a set $N=\Sigma \cap M^{c}$ of $\mathcal{H}^{n}$-measure zero such that:
$$\Sigma= M \cup N.$$
In particular, $\Sigma$ is $n$-rectifiable.

\end{proof}

 
 \begin{definition}\label{normalized} Let $\mu$ be an $n$-uniform measure in $\RR^{d}$, $x_0 \in \mbox{supp}(\mu) \cup \left\lbrace \infty \right\rbrace$. We will call $\mu^{x_0}$ the normalized tangent measure to $\mu$ at $x_0$ if $\mu^{x_0} \in \mbox{Tan}(\mu,x_0)$, and $\mu^{x_0}(B(0,1))=\omega_n$.
 \end{definition}
 
One of the most remarkable results in Preiss' paper $\cite{P}$ is a separation between flat and non-flat measures at infinity. We will state a reformulation of this theorem by De Lellis from $\cite{Del}$ which is better adapted to our needs.

\begin{theorem}[\cite{P}]\label{flatnessinfty}
Let $\mu$ be an  $n$-uniform measure in $\RR^{d}$, $\zeta$ its normalized tangent at $\infty$  (in the sense of Definition $\eqref{normalized}$). If $n \geq 3$, then there exists $\epsilon_{0}>0$ (depending only on $n$ and $d$) such that, if \begin{equation}
\min_{V \in G(n,d)} \int_{B(0,1)} dist^{2}(z,V) d\zeta(z) \leq \epsilon_{0},
\end{equation}
then $\mu$ is flat.

In particular, if $\mu$ is conical and \begin{equation} \min_{V \in G(n,d)} \int_{B(0,1)} dist^{2}(z,V) d\mu(z) \leq \epsilon_{0},
\end{equation} then $\mu$ is flat.
\end{theorem}

$\cite{Del}$  defines certain functionals that measure how far from flat a measure is and behave well under weak convergence.

\begin{definition} \label{functional}
Let $\varphi \in C_{c}(B(0,2))$, $0 \leq \varphi \leq 1$ and $\varphi=1$ on $B(0,1)$.
We define the functional $F: \mathcal{M}(\RR^{d}) \to \RR$ as
$$ F(\Phi): = \min_{V \in G(n,d)} \int \varphi(z) {dist}^{2}(z,V)d\Phi(z)$$
\end{definition}

\begin{lemma}[\cite{Del}]\label{contfunctional}
Let $\Phi_{j}$ , $\Phi$ be Radon measures such that $\Phi_{j} \rightharpoonup \Phi$.
Then $F(\Phi_{j}) \to F(\Phi)$. 
\end{lemma}

We can now reformulate Theorem $\ref{flatnessinfty}$ in terms of the functionals $F$.

\begin{corollary}\label{functionalflatness}
Let $\mu$ be an $n$-uniform measure on $\RR^{d}$, $\zeta$ its normalized tangent at infinity. If $n \geq 3$, there exists $\epsilon_{0}>0$ (depending only on $n$ and $d$) such that 
$$F(\zeta) \leq \epsilon_{0} \implies \mu \mbox{ is flat}.  $$
In particular, if $\mu$ is conical and $F(\mu) \leq \epsilon_{0}$ then $\mu$ is flat.
\end{corollary}
\begin{proof}
By definition of $\varphi$, we have: 
$$\chi_{B(0,1)}(x) \leq \varphi(x),$$
for all $x \in \RR^{d}$.
This implies that
$$\min_{V \in G(n,d)} \int_{B(0,1)} dist^{2}(z,V) d\zeta(z)  \leq F(\zeta),$$
and in particular, if $\epsilon_{0}$ is the constant from Theorem $\ref{flatnessinfty}$
\begin{align*}
F(\zeta) \leq \epsilon_{0} & \implies \min_{V \in G(n,d)} \int_{B(0,1)} dist^{2}(z,V) d\zeta(z) \leq \epsilon_{0}, \\ & \implies \mu \mbox{ is flat. }
\end{align*}
This ends the proof.
\end{proof}

Another result concerning the geometry of supports of uniformly distributed measures was proven in $\cite{KiP}$, with the added condition that their support be bounded. This result states that in this case, the support is in fact an algebraic variety.
\begin{theorem}[\cite{KiP}]\label{boundedvariety}
Let $\mu$ be a uniformly distributed measure over $\RR^{d}$ with bounded support and let $u \in \Sigma$. Then $x \in \Sigma$ if and only if:
\begin{equation}\label{boundedsupport} P_{k}(x)= \int_{\RR^d} {\left\langle z-x, z-x \right\rangle}^{k} - {\left\langle z-u,z-u \right\rangle}^{k} d\mu(z)=0, 
\end{equation}
for every $k \in \NN$.
\end{theorem}

\subsection{Measures with locally $C^{\alpha}$ density ratio and their geometry}
In $\cite{DKT}$ and $\cite{PTT}$, the authors proved that for measures with nice density ratios, Theorem $\cite{P}$ can be improved in the sense that the support is a $C^{1,\beta}$-manifold in the neighborhood of every flat point, for some $\beta >0$.
Let us start with some definitions.
For $x \in \Sigma$  where $\Sigma$ is a closed set and $r>0$, set: 
\begin{equation}
\theta(x,r)= \frac{1}{r} \inf \left\lbrace D\left[\Sigma \cap \overline{B(x,r)},L\cap \overline{B(x,r)} \right] : L \text{ affine n-plane through }x \right\rbrace ,
\end{equation}
 where,
 \begin{equation} 
 D[E,F]=\sup\left\lbrace dist(y,F) : y \in E\right\rbrace + \sup \left\lbrace dist(y,E) : y \in F \right\rbrace
 \end{equation}
 denotes the Hausdorff distance between the closed sets $E$ and $F$.

\begin{definition}
Let $\delta >0$ be given. We say that the closed set $\Sigma \subset \RR^{d}$ is $\delta$-Reifenberg-flat of dimension $n$ if for all compact sets $K \subset \Sigma$ there is a radius $r_K >0$ such that:
\begin{equation}
\theta(x,r) \leq \delta \text{ for all } x \in K \text{ and } 0<r \leq r_{K}.
\end{equation}

\end{definition}
\begin{definition}
We say that the closed set $\Sigma$ is Reifenberg flat with vanishing constant of dimension $n$ if for every compact subset $K$ of $\Sigma$: 
\begin{equation}
\lim_{r \to 0^{+}} \theta_{K}(r)=0 
\end{equation}
where 
\begin{equation}
\theta_{K}(r)=\sup_{x \in K} \theta(x,r) .
\end{equation}
\end{definition}

\begin{definition}
Let $\mu$ be a Radon measure on $\RR^d$.
\begin{itemize}
\item  We say that $\mu$ has $n$-density ratio locally $C^{\alpha}$  if, for each compact set $K \subset \Sigma$, there is a constant $C_{K}$ such that:
\begin{equation}\label{PTTdoublingrate}
\left|\frac{\mu(B(x,r))}{{\omega}_n r^n}-1 \right| \leq C_{K} r^{\alpha},
\end{equation}
for $x \in K$ and $0<r<1$.

\item If $x \in \Sigma$, $r>0$ and $t \in (0,1]$, define the quantity:
\begin{equation}
R_{t}(x,r)= \frac{\mu(B(x,tr))}{\mu(B(x,r))} - t^{n}.
\end{equation}
We say $\mu$ is asymptotically optimally doubling if for each compact set $K \subset \Sigma$, $x \in K$, and $t \in [\frac{1}{2},1]$ 
\begin{equation}\label{asymptoptim}
\lim_{r \to 0^{+}} \sup_{x \in K} \left| R_{t}(x,r) \right| = 0.
\end{equation}
\end{itemize}
\end{definition}
The following results from $\cite{PTT}$ and $\cite{L}$ describe the geometry of the support of a  measure based on information on its density ratio. 

\begin{theorem}[1.9, \cite{PTT}] \label{PTT1}
For each $\alpha \in (0,1]$, there exists $\beta=\beta(\alpha)>0$ with the following property. Suppose $\mu$ is a positive Radon measure supported on $\Sigma \subset \RR^{d}$ and for each compact set $K \subset \Sigma $ there exists a constant $C_K$ such that:
\begin{equation}\label{doublingratio}
|R_{t}(x,r)|\leq C_{K} r^{\alpha} \mbox{ for } r \in (0,1], t \in [\frac{1}{2},1] \mbox{ and } x \in K.
\end{equation}
 Then: \begin{itemize}
\item If $n =1,2$,  $\Sigma$ is a $C^{1,\beta}$-submanifold of dimension $n$ in $\RR^{d}$.
 \item   If $n\geq 3$ there exists a constant $\delta(n,d)$ such that if $x_0 \in \Sigma$ and $\Sigma \cap B(x_0,2R_0)$ is $\delta$-Reifenberg-flat, then $\Sigma \cap B(x_0,R_0)$ is a $C^{1,\beta}$-submanifold of dimension $n$ in $\RR^{d}$.
 \end{itemize}
\end{theorem}
 \begin{theorem} [\cite{PTT}, \cite{L}] \label{PTT2}
 
For each $\alpha >0$, there exists $\beta=\beta(\alpha)$ with the following property. If $\mu$ is a positive Radon measure supported on $\Sigma \subset \RR^d$ whose $n$-density ratio is locally $C^{\alpha}$, then:
\begin{itemize}
\item(1.10, $\cite{PTT}$) if $n=1,2$, $\Sigma$ is a $C^{1,\beta}$ submanifold of dimension $n$ in $\RR^{d}$.
\item(1.10, $\cite{PTT}$) if $n \geq 3$, $\Sigma$ is a $C^{1,\beta}$ submanifold of dimension $n$ in $\RR^{d}$ away from a closed set $\mathcal{S}$ such that $\mathcal{H}^{n}(\mathcal{S})=0$, where  $\mathcal{S}=\Sigma \backslash \mathcal{R}$ and $\mathcal{R}= \left\lbrace x \in \Sigma ; \limsup_{r \to 0} \theta(x,r) = 0 \right\rbrace$.
\item (1.7, $\cite{L}$) If $n=3$, $d=4$, and $x$ is a non-flat point of $\Sigma$, there exists a neighborhood of $x$ which is $C^{1,\beta}$ diffeomorphic to an open piece of the KP-cone $C$ in $\eqref{KPcone}$, containing the singular point $0$.
\end{itemize}
 \end{theorem}
 
 \begin{theorem}[III.5.9, \cite{DS}] \label{weakconv2}
 Let $\left\lbrace \mu_{j} \right\rbrace_{j}$ be a sequence of $n$-uniform measures with constant $c$, converging weakly to a Radon measure $\lambda$. Then for every ball $B \subset \RR^{d}$, we have: 
 $$ \lim_{j \to \infty} \left(\sup_{x \in B \cap \mbox{supp}(\lambda)} \mbox{ dist}(x,\mbox{supp}\mu_j) \right) = 0 $$ and
 $$\lim_{j \to \infty} \left(\sup_{x \in B \cap \mbox{supp}(\mu_j)} \mbox{ dist}(x,\mbox{supp}\lambda) \right) = 0.$$
 \end{theorem}
As a corollary of Theorem $\ref{PTT2}$ and Theorem $\ref{weakconv2}$, we get:

\begin{corollary}\label{consPTT2}[\cite{PTT}]
Let $\mu$ be an $n$-uniform measure in $\RR^{d}$, let $\Sigma$ be its support and $x \in \Sigma$ a flat point. Then there exists $R >0$ depending on $x$, $n$, $d$, $\mu$ and $\beta$  such that $\Sigma \cap B(x,R)$ is a $C^{1,\beta}$ $n$-submanifold.
\end{corollary}
\begin{proof}
It is clear from Theorem $\ref{PTT2}$ that we only need to prove that every flat point of $\mu$ is in $\mathcal{R}$, namely that $\lim_{r \to 0} \theta(x,r) =0$ for such an $x$. 
But taking $\mu_{j}$ to be $\mu_{x,r_j}$ and $\lambda= \mathcal{H}^{n} \res V$, where $V$ is the tangent plane at $x$, the result follows directly from Theorem $\ref{weakconv2}$.
\end{proof}

\newpage
\section{The base case: singularities of 3-uniform conical measures}
We first study the case where $\nu$ is a conical $3$-uniform measure. 
We start by proving that $\nu$ decomposes into a uniformly distributed spherical component and a radial component. 

Let $\nu$ be a conical $n$-uniform measure in $\mathbb{R}^d$, with $0$ in its support. Let $\Sigma$ be the support of $\nu$. 
In particular, since $\nu$ is conical, for any $r>0$, we have by Theorem $\ref{momentsconical}$
$$\Sigma = r \Sigma.$$

By Theorem $\ref{supportrect}$, normalizing $\nu$ if necessary,
 $$\nu= \mathcal{H}^n \res \Sigma.$$

\begin{definition}
Let $\nu$ be a conical $n$-uniform measure in $\mathbb{R}^d$, with $0$ in its support, $\Sigma$ its support.
We define $\sigma$ to be the spherical component of $\nu$, namely:
$$\sigma = \mathcal{H}^{n-1} \res ({\Sigma \cap {S^{d-1}}}),$$
where $S^{d-1}= \left\lbrace x \in \RR^{d} ; |x|=1 \right\rbrace$.
\end{definition}

\begin{definition}  Let $E \subset S^{d-1}$ and $\rho>0$. We define $\rho E$ the dilate of $E$ by: $$ \rho E = \{y \in \mathbb{R}^d ; \frac{y}{|y|} \in E, |y|= \rho \}. $$
We define $E_{r}^{\delta}$, the $(r,\delta)$- neighborhood of $E$, or $\delta$-neighborhood of $r E$  to be: 
\begin{equation} \label{polarnhood} E_{r}^{\delta} = \left\lbrace y; \frac{y}{|y|} \in E, |y| \in (r(1-\delta), r(1+\delta)) \right\rbrace. 
\end{equation} \end{definition}

Our first goal is to prove a polar decomposition for $\nu$, namely that $\nu$ decomposes into its spherical component and a radial component.

\begin{theorem}\label{polar} Let $\nu$ be a conical $n$-uniform measure in $\RR^{d}$. Let $g$ be a Borel function on $\RR^{d}$. Then:
\begin{equation}\label{eqpolar}
\int g(x) d\nu(x)= \int_{0}^{\infty} \rho^{n-1} \int g(\rho x') d\sigma(x') d\rho, 
\end{equation} 
where $\rho=|x|$ and $x'=\frac{x}{|x|}$. \end{theorem}
\begin{proof}
Let $u: \mathbb{R}^d \to \mathbb{R}_+$ be the function given by: $u(x)=|x|$.
Our first aim is to prove that for any $g=\chi_{A}$ where  $A \subset \RR^{d}$ is a Borel set, we have:
\begin{equation} \label{polarstep1}
\int g(x) d\nu(x)= \int_{0}^{\infty} \rho^{n-1} \int g(\rho x') d\sigma(x') d\rho.
\end{equation}

Note that if $A$ is a Borel set, $\frac{A}{\rho} \cap S^{d-1}$ being the intersection of the pre-image of $A$ by the dilation homeomorphism, with the Borel set $S^{d-1}$, is also Borel.

Now, since $u$ is Lipschitz (in fact smooth away from 0), and $J_{\Sigma}^{*}u=|\nabla^{\Sigma}u|$  we can apply the co-area formula  $\eqref{coarea}$  to the rectifiable set $\Sigma$, the Lipschitz function $u$ and the Borel function $\chi_{A}$ to get:
\begin{equation}\label{coarea1} \int_{A \cap \Sigma} |\nabla^{\Sigma}u|(y) d\mathcal{H}^n(y)= \int_{0}^{\infty}  \int_{u^{-1}(\rho)} \chi_{A \cap \Sigma}(y) d\mathcal{H}^{n-1}(y) d\rho. 
\end{equation}
Note that by Theorem $\ref{PTT2}$, away from a closed set $\mathcal{S}$ of $\mathcal{H}^n$ measure 0, $\Sigma$ is a $C^{1,\alpha}$ submanifold of dimension $n$. Therefore, in ($\ref{coarea1}$) , we can define $\nabla^{\Sigma}u$ to be the gradient in the manifold sense at almost every point.

We first claim that $|\nabla^{\Sigma}u|(x)=1$ for almost every $x \in \Sigma$. Let $x$ be a flat point of $\Sigma$ (namely a point where $\nu$ admits a unique flat tangent). We can take $\tau_x=\frac{x}{|x|}$ to be an element of an orthonormal basis of $P_x$. Indeed, $x$ being a flat point, by Corollary $\ref{consPTT2}$, $\Sigma$ is a $C^{1,\beta}$-manifold in a neighborhood of $x$, and the tangent space at $x$ is $P_x$. 
Now consider the curve $\gamma(t)=t \tau_x +x$. 
Since $x \in \Sigma$ and $\nu$ is conical, $\gamma \in \Sigma$, $\gamma(0)=x$ and $\gamma'(0)=\tau_x$.
Complete the unit vector $\tau_1=\tau_{x}$ to a full orthonormal basis $\{\tau_i\}_{i=1}^{n}$ of $\RR^{d}$. We have $\nabla u(x)= \tau_x$. Therefore: $\nabla^{\tau_x} u = \tau_x \;.\; \tau_x =1$ and $\nabla^{\tau_j} u = \tau_x \; . \; \tau_j = 0$ for $j>1$, by construction of the basis.  Since almost every point of $\Sigma$ is flat, this proves that $|\nabla^{\Sigma}u|= 1$ almost everywhere, proving the claim.
 Moreover, if $E \subset S^{d-1}$ is a Borel set, since $\nu$ is conical and $\Sigma= \frac{\Sigma}{\rho}$, we have:
 \begin{align}\label{Hausdorffdilation}
 \mathcal{H}^{n-1}(\rho E \cap \Sigma) & = \mathcal{H}^{n-1}\left(\rho\left(E \cap \frac{\Sigma}{\rho}\right)\right), \nonumber \\
 &=\rho^{n-1} \mathcal{H}^{n-1}\left(E \cap \frac{\Sigma}{\rho}\right), \\
 &= \rho^{n-1} \mathcal{H}^{n-1}\left(E \cap \Sigma\right). \nonumber
 \end{align}
 
Therefore,
\begin{align}
\label{polarsubstep1.1}
\int_{A \cap \Sigma}  d\mathcal{H}^{n}(y)
&=\int_{A \cap \Sigma} J^{*}u(y) d\mathcal{H}^{n}(y), \nonumber\\
& =\int_{0}^{\infty}  \int_{u^{-1}(\rho)} \chi_{A \cap \Sigma}(y) d\mathcal{H}^{n-1} (y) d\rho \nonumber,\\
&=\int_{0}^{\infty} \mathcal{H}^{n-1}(\Sigma \cap A \cap {{\partial B}_{\rho}}) d\rho,\\
&=\int_{0}^{\infty} \rho^{n-1} \mathcal{H}^{n-1}(\frac{\Sigma}{\rho} \cap \frac{A}{\rho} \cap {{\partial B}_{1}}) d\rho ,\nonumber \\
&=\int_{0}^{\infty} \rho^{n-1} \mathcal{H}^{n-1}(\Sigma \cap \frac{A}{\rho} \cap {{\partial B}_{1}} )d \rho , \text{ since } \Sigma \text{ is conical, } \nonumber \\
&=\int_{0}^{\infty} \rho^{n-1} \int \chi_{A}(\rho z') d\sigma(z') d\rho . \nonumber
\end{align}

Now let $g$ be a non-negative Borel function. Then there exists an increasing sequence of simple functions $\left\lbrace g_k \right\rbrace$ converging pointwise to $g$. In particular, if $g_k$ increase to $g$ pointwise, then $G_k$ increase pointwise to $G$ where $G_k (\rho) = \rho^{n-1}\int g_{k}(\rho x') d\sigma(x')$ and $G(\rho) = \rho^{n-1}\int g(\rho x') d\sigma(x')$. By the monotone convergence theorem and linearity of the integral, $\eqref{eqpolar}$ also holds for non-negative Borel functions. The extension to general Borel functions follows easily.
\end{proof}

Having proven that $\nu$ decomposes into two components, we now study the spherical component $\sigma$. By using the polar decomposition, we can prove that $\sigma$ is uniformly distributed. Of particular interest to us is the case where $\nu$ is $3$-uniform: the spherical component is then locally $2$-uniform.
We start with some notations.
Denote $\Sigma \cap S^{d-1}$ by $\Omega$. Then $\sigma = \mathcal{H}^{n-1} \res \Omega$. 

Let $B_r(x)= \left\lbrace z \in S^{d-1} ; |z-x| <r \right\rbrace$, and 
$\left(B_r(x)\right)_{1}^{\epsilon}$ be as in ($\ref{polarnhood}$).

  \begin{theorem}\label{spherunifdist} Let $\nu$ be as in Theorem $\ref{polar}$. Then $\sigma$ the spherical component of $\nu$  is a uniformly distributed measure. \end{theorem}
  \begin{proof} Let $x \in \Omega$, $r>0$. Define the set $N_{\left\lbrace|x|,r\right\rbrace} \subset [0,\infty) \times \mathbb{R}$ to be:
$$N_{\left\lbrace|x|,r\right\rbrace}=\left\lbrace(a,b) \in  [0,\infty) \times \mathbb{R} ; (1+|x|^2-r^2)a -2b <0 \right\rbrace.$$
Then:
\begin{align*}z \in \left( B_r(x) \right)_{1}^{\epsilon} & \iff \left| \frac{z}{|z|}-x \right|^2 < r^2 \mbox{ and } |z| \in (1-\epsilon, 1+ \epsilon)\\
& \iff \left( 1+|x|^2-r^2\right)|z|-2 \left\langle z,x \right\rangle <0 \mbox{ and } |z| \in (1-\epsilon, 1+\epsilon), \end{align*}
allowing us to rewrite $g(z)=\chi_{\left( B_r(x) \right)_{1}^{\epsilon}}(z)$ in the following way:
\begin{equation}
g(z)= \chi_{N_{{\left\lbrace|x|,r\right\rbrace}}}(|z|, \left\langle z,x \right\rangle) . \chi_{(1-\epsilon, 1+\epsilon)}(|z|) =G(|z|, \left\langle z,x \right\rangle).
\end{equation}
  
 Since $\nu$ is a conical uniform measure and $g$ is a function of $|z|$ and $\left\langle z,x \right\rangle$ with $x$ in the support of $\nu$, we can apply  Theorem $\ref{momentsconical}$ to it. Namely, fix $e \in \mathbb{R}^n$, $|e|=1$. Then by Theorem $\ref{momentsconical}$ and polar decomposition for Lebesgue measure:
  \begin{align} 
  \nu ( \left( B_r(x) \right)_{1}^{\epsilon} ) &= \int G(|z|,<z,x>) d\nu(z) \nonumber
  \\ &=\int_{\RR^{n}}  G(|z|,<z,e>) d\mathcal{L}^{n}(z), \nonumber
\\&= \int_{0}^{\infty} \left( \int_{\left\lbrace |y|=\rho \right\rbrace} \chi_{B_r(e)}\left(y'\right) \chi_{(1-\epsilon,1+\epsilon)}\left(|y|\right)d\mathcal{H}^{n-1}(y) \right) d\rho, \mbox{ where } y'=\frac{y}{|y|}, \nonumber
\\&= \int_{0}^{\infty} \rho^{n-1} \int_{S^{n-1}} \chi_{B_{r}(e)}(y') \chi_{(1-\epsilon, 1+\epsilon)}(\rho) d\sigma^{n-1}(y') d\rho
\\&=\left(\int_{1-\epsilon}^{1+\epsilon} {\rho}^{n-1}  d\rho \right)\left(\mathcal{H}^{n-1}(B_r(e) \cap S^{n-1})\right),\nonumber
\\ &= \frac{(1+\epsilon)^{n}- (1-\epsilon)^{n}}{n} \left( \mathcal{H}^{n-1}(B_r(e) \cap S^{n-1})\right). \nonumber
   \end{align}
  where $\mathcal{L}^{n}$ is $n$-Lebesgue measure. 
Dividing by $2 \epsilon$ and letting $\epsilon$ go to $0$ gives:
\begin{equation} \label{Lebesgue}\mathcal{H}^{n-1}(B_r(e) \cap S^{n-1}) = \lim_{\epsilon \to 0} \frac{1}{2\epsilon}  \nu ( \left( B_r(x) \right)_{1}^{\epsilon} ). 
\end{equation}

  Note that $\mathcal{H}^{n-1}(B_r(e) \cap S^{n-1}) $ does not depend on our choice of $e$.
  
  On the other hand, by Theorem $\ref{coarea}$ and $\eqref{Hausdorffdilation}$, we get:
  \begin{align}
  \nu(\left( B_r(x) \right)_{1}^{\epsilon}) &= \int_{0}^{\infty} \left( \int_{\Sigma \cap \left\lbrace |y|=\rho \right\rbrace} \nu_{B_r(x)}\left(\frac{y}{|y|}\right) \nu_{(1-\epsilon,1+\epsilon)}(|y|)d\mathcal{H}^{n-1}(y) \right) d\rho, \nonumber
  \\&=\int_{1-\epsilon}^{1+\epsilon} \mathcal{H}^{n-1}(\rho B_r(x) \cap \Sigma) d\rho, \nonumber
  \\&= \int_{1-\epsilon}^{1+\epsilon} {\rho}^{n-1} \mathcal{H}^{n-1}(B_r(x) \cap \Sigma) d\rho, 
  \\&=\left(\int_{1-\epsilon}^{1+\epsilon} {\rho}^{n-1}  d\rho \right) \left(\mathcal{H}^{n-1}(B_r(x) \cap \Sigma) \right), \nonumber
  \\&=\frac{(1+\epsilon)^{n}- (1-\epsilon)^{n}}{n} \left(\mathcal{H}^{n-1}(B_r(x) \cap \Sigma) \right). \nonumber
  \end{align} 
   Dividing by $2 \epsilon$ and letting  $\epsilon$ go to $0$ gives:
   \begin{equation} \label{notLebesgue} 
   \mathcal{H}^{n-1}(B_r(x) \cap \Sigma)  =  \lim_{\epsilon \to 0} \frac{1}{2\epsilon}  \nu ( \left( B_r(x) \right)_{1}^{\epsilon} ).
   \end{equation}
   Combining $\eqref{Lebesgue}$ and $\eqref{notLebesgue}$, we get:
   \begin{equation} \label{sphericalunifdist}
    \sigma(B_r(x))=\mathcal{H}^{n-1}(B_r(x) \cap \Omega)=\mathcal{H}^{n-1}(B_r(e) \cap S^{n-1}), \text{for any } x \in \Omega, \text{ and any } e \in S^{n-1}.
    \end{equation}

  In particular, this implies that $\sigma$ is uniformly distributed. \end{proof}
  One notable consequence of the above, expressed in the following corollary, is that for $n=3$, the spherical component is in fact locally $2$-uniform.
  \begin{corollary}\label{3unif} Suppose $\nu$ a $3$-uniform conical measure on $\RR^d$. Let $\sigma$ be its spherical component, and denote the support of $\sigma$ by $\Omega$. Then there exists a function $\phi: \RR_{+} \rightarrow \RR_{+} $ such that, for all $x \in  \Omega$, for all $r>0$:
  \begin{equation}\label{3unif1} \sigma(B(x,r))=\phi(r). 
  \end{equation} 
Moreover,   
  \begin{equation} \label{3unif2}
  \phi(r)=\pi r^{2} \chi_{(0,2)}(r)+4 \pi \chi_{2,\infty}(r). 
  \end{equation} \end{corollary}
  
  \begin{proof} $\eqref{3unif1}$ is just a reformulation of Theorem $\ref{spherunifdist}$.
    Let $e=(0,0,1)$ .
    We only need to prove that for $r<2$, we have:
  $$ \mathcal{H}^{2}(B_r(e) \cap S^{2}) =\pi r^2.$$
  First, note that $\partial B_{r}(e) \cap S^{2} = \{(x,y,z) \in \RR^3; x^2+y^2+z^2=1, x^2+y^2+(z-1)^2=r^2 \}$. If $r<\sqrt{2}$, $B_{r}(e) \cap S^{2}$ is the portion of the graph of $f(x,y)=\sqrt{1-(x^2+y^2)}$ above $z=1-\frac{r^2}{2}$. So we have, by the area formula:
  
  \begin{align*} \mathcal{H}^{2}(B(e,r) \cap S^{2}) &= \int_{0}^{2\pi} \int_{0}^{\sqrt{1-(1-\frac{r^2}{2})^2} } \sqrt{1+|\nabla{f}|^2} \rho d\rho d\theta\\
                                                                                         &= 2\pi \int_{0}^{\sqrt{1-(1-\frac{r^2}{2})^2} }  \frac{\rho}{\sqrt{1-{\rho}^2}} d\rho \\
                                                                                         &= -2 \pi \left( \sqrt{1-\left(1-\left(1-\frac{r^2}{2}\right)^2\right)}-1 \right) \\
                                                                                         &=\pi r^2. \end{align*}
                                                                                      
                                                                                      If $ \sqrt{2}<r <2$, $B(e,r)$ and $B(0,1)$ intersect in $z=1-\frac{{r}^2}{2}$.
                                                                                       Moreover, note that the part of $S^2$ below the plane $z=1-\frac{{r}^2}{2}$ is $B(-e,r')$, where, by applications of Pythagoras' theorem, we have: \begin{align*}
{r'}^2 &= 1+\left(2-\frac{{r}^2}{2}\right)^{2}- \left(\frac{r^{2}}{2}-1\right)^{2}, \\
              &=4-r^{2}.
\end{align*} 
Therefore, by symmetry (since $r' < \sqrt{2}$), we have:
\begin{align*}
\mathcal{H}^2(B(e,r) \cap {S^2}) &=\mathcal{H}^2(S^2) - \mathcal{H}^2(B(-e,r')\cap S^2), \\
&=4\pi - \pi(4-r^{2}), \\
&=\pi r^{2}. 
\end{align*}

  \end{proof}

In [KiP], Kircheim and Preiss had proved that the support of a uniformly distributed measure is an analytic variety. We will now deduce from Corollary $\ref{3unif}$ that $\Sigma$ is in fact an algebraic variety. 

Recall from $\eqref{conicalevenodd}$ that: 
$$
b_{2k}^{k} = k! b_{2k,1} \mbox{ and } b_{2k-1,1}=0,
$$
 where $$ b_{k,1}(x)= 2^{k}(I(1) (k)!)^{-1} \int \left\langle z, x \right\rangle^{k} e^{-|z|^{2}} d\mu(z) \mbox{ and } I(s)= \int e^{-s|z|^2} d\nu(z).$$

 Applying Theorem $\ref{polar}$ to $b_{k,1}$ and $I(1)$ gives:
 
 \begin{align*}
 b_{k,1}(x)& = 2^{k}(I(1) k!)^{-1} \int \left\langle z, x \right\rangle^{k} e^{-|z|^{2}} d\nu(z),\\
 &= 2^{k}(I(1) k!)^{-1}  \int_{0}^{\infty} {\rho}^{n-1} \int \left\langle \rho z' , x \right\rangle^{k} e^{-{\rho}^{2}} d\sigma(z') d\rho, \\
 &= 2^{k}(\sigma(S^{d-1})k!)^{-1}  \frac{\int_{0}^{\infty} {\rho}^{n-1+k} e^{-{\rho}^2} d\rho}{\int_{0}^{\infty} {\rho}^{n-1} e^{-{\rho}^2} d\rho} \int \left\langle  z' , x \right\rangle^{k}  d\sigma(z'), \\
   &=c(n,k)  \int \left\langle  z' , x \right\rangle^{k}  d\sigma(z'),    
 \end{align*}
 where $c(n,k)=2^{k}(\sigma(S^{d-1})k!)^{-1}  \frac{\int_{0}^{\infty} {\rho}^{n-1+k} e^{-{\rho}^2} d\rho}{\int_{0}^{\infty} {\rho}^{n-1} e^{-{\rho}^2} d\rho}$.
 Therefore: $b_{2k}^{k}(x) = k! c(n,k) \int \left\langle  z' , x \right\rangle^{k}  d\sigma(z')$.
 
We can now improve Theorem $\ref{momentsconical}$ in the case of the spherical component $\sigma$ of a conical $n$-uniform measure $\nu$: indeed, in this case $\Omega$ is entirely described by its moments.
\begin{theorem} \label{algVar}  Let $\nu$ be a conical $n$-uniform measure, $\sigma$ its spherical component, $p_{2k}(x)= b_{2k}^{k}(x)$ and $p_{2k+1}(x)= b_{2k+1,1}(x)$. Moreover, let $\Omega$ be the support of $\sigma$.
 Then \begin{equation}\label{odd}
 p_{2k+1}(x)=0 \text{ for } k \geq 0, x \in \mathbb{R}^d
\end{equation}
 
 and \begin{equation} \label{even}
 \Omega = \left\lbrace x; |x|=1 \right\rbrace \cap \bigcap_{k>0} \left\lbrace x; p_{2k}(x)=|x|^{2k} \right\rbrace.
 \end{equation}
 \end{theorem}
 \begin{proof}
 Call $\Omega'$ the right-hand side of $\eqref{even}$. The fact that $\Omega \subset \Omega'$ and $\eqref{odd}$ follows from Theorem $\ref{momentsconical}$.
  
 To prove the other inclusion, take any $x \in \mathbb{R}^d$ such that $x \in \Omega' $.
 Let us rewrite the expressions from Theorem $\ref{boundedvariety}$.
 First, note that since $|x|=1$: \begin{align*}
 \langle z-x,z-x \rangle ^{l}& =(|z|^2+|x|^2-2 \langle z,x \rangle )^{l} \\
                                  &=\sum_{0}^{l} \binom{l}{i} (-1)^{i} (|z|^2+1)^{l-i} 2^{i} \langle z,x \rangle ^{i}.
 \end{align*}
 Moreover,using $\eqref{odd}$:
 \begin{align*}
  \int \langle z-x,z-x \rangle ^{l} d\sigma(z) 
 &= \sum_{0}^{l} \binom{l}{i} (-1)^{i}  2^{i} \int (|z|^2+1)^{l-i}\langle z,x \rangle ^{i} d\sigma(z),\\
 &=\sum_{0}^{l} \binom{l}{i} (-1)^{i}  2^{l} \int \langle z,x \rangle ^{i} d\sigma(z), \\
 &=\sum_{k; 2k \leq l} \binom{l}{2k}   2^{l} \int \langle z,x \rangle ^{2k} d\sigma(z),\\
 &=\sum_{k; 2k \leq l} \binom{l}{2k}   2^{l}  (k! c(n,2k))^{-1} p_{2k}(x)\\
 &=\sum_{k; 2k \leq l} \binom{l}{2k}   2^{l}  {k! c(n,2k)}^{-1}, \mbox{    since }p_{2k}(x)=|x|^{2k}=1\mbox{ by hypothesis}.
          \end{align*}
Note that we have proved that for all $x  \in \Omega'$,  $\int \langle z-x,z-x \rangle ^{l} d\sigma(z)  = c_{l}$ where $c_l$ does not depend on the choice of $x$. In particular, if $u \in \Omega$ is fixed as in Theorem $\ref{boundedvariety}$, since $\Omega \subset \Omega'$, we have:
\begin{equation}
\int \langle z-x,z-x \rangle ^{l} - \langle z-u,z-u \rangle ^{l} d\sigma(z)= c_{l}-c_{l} =0. 
\end{equation}
This proves that $x \in \Omega$ by Theorem $\ref{boundedvariety}$.

 \end{proof}
 As an easy consequence of the above claim, we get:
 
 \begin{corollary} \label{algvarcone}
 Let $\nu$ be a conical $n$-uniform measure in $\mathbb{R}^d$ and $\Sigma$ its support. Then $\Sigma$ is an algebraic variety and \begin{equation}
 \Sigma=-\Sigma.
 \end{equation}
 \end{corollary}
 \begin{proof}
 Let $\Sigma '$ be $\bigcap_{k>0} \left\lbrace b^{k}_{2k}(x) =|x|^{2k} \right\rbrace$. By $(\ref{varietyconical})$, $\Sigma \subset \Sigma'$.
 Now suppose $x \notin \Sigma$. Then $x \neq 0$ and $\frac{x}{|x|} \notin \Omega$, where $\Omega$ is the support of the spherical component of $\nu$. By Theorem $\ref{algVar}$, there exists $k_0$ such that : $p_{2k_0}(\frac{x}{|x|})\neq 1$.
 
 Multiplying by $|x|^{2k_0}$, we get:
 $ b^{k}_{2k}(x) \neq |x|^{2k_0},$ and hence $x \notin \Sigma'$.

By Theorem $\ref{momentsconical}$, for a conical measure, $b_{2k}^{k}(x)=b_{2k,1}(x)$ which is a homogeneous polynomial of even degree. In particular
\begin{align*} x \in \Sigma & \iff b^{k}_{2k}(x) =|x|^{2k}, \mbox{ }k \in \NN \\
 & \iff b^{k}_{2k}(-x) =|-x|^{2k},\mbox{ } k \in \NN \\ & \iff -x \in \Sigma.
  \end{align*}
Hence \begin{equation} \label{coneeven}
\Sigma = -\Sigma.
\end{equation}
\end{proof}

 We will apply Theorem $\ref{PTT1}$ to deduce that $\Omega$ is a $C^{1,\beta}$ manifold for some $\beta >0$. We then prove that there exists $\gamma>0$ such that $\Sigma \backslash \left\lbrace0\right\rbrace$ is a $C^{1,\gamma}$-submanifold of dimension $3$ in $\mathbb{R}^d$.

 \begin{lemma}\label{SpherCompMan} Let $\nu$ be a $3$-uniform conical measure in $\mathbb{R}^d$, $\sigma$ its spherical component and $\Omega$ the support of $\sigma$.
 Then there exists $\beta >0$ such that $\Omega$ is a $C^{1,\beta}$ submanifold of dimension $2$ in $\mathbb{R}^d$. \end{lemma}
 
 \begin{proof} According to Theorem $\ref{PTT1}$, we only need to prove $\eqref{doublingratio}$ for $\sigma$. Let $r \in (0,1]$ and $t \in [\frac{1}{2},1]$ so that $tr \in (0,1]$. By Corollary $\ref{3unif}$, for any $x \in \Omega$, we have $\frac{\sigma(B(x,tr))}{\sigma(B(x,r))} = \frac{t^2 r^2}{r^2}=t^2$ implying that :
 $$|R^{2}_{t}(x,r)|=0, \text{ for } r \in (0,1], t \in [\frac{1}{2},1], \text{ and } x \in \Omega.$$
 \end{proof}
 
 \begin{theorem}\label{onesingularity} Let $\nu$ be a $3$-uniform conical measure in $\mathbb{R}^d$ and $\Sigma$ its support. Then there exists $\gamma>0$ such that $\Sigma \backslash \left\lbrace0\right\rbrace$ is a $C^{1,\gamma}$ submanifold of dimension $3$ . \end{theorem}
 
  \begin{proof} 
By dilation invariance of $\Sigma$, it is enough to prove that $\Sigma$ is a $C^{1,\gamma}$-manifold in a neigborhood of $x_0 \in \Omega$.
  Let $\sigma$ be the spherical component of $\nu$. By Lemma $\ref{SpherCompMan}$, $\Omega$ is a $C^{1,\beta}$ submanifold of dimension $2$ in $\mathbb{R}^d$. In particular, fix $x_0 \in \Omega$. There exists a $2$-subspace $P_{x_0}$ of ${\RR}^d$ tangent to $\Omega$ at $x_0$. Let $\left\lbrace \tau_1, \tau_2 \right\rbrace$ be an orthonormal basis of $P_{x_0}$.
 Since $\Omega \subset S^{d-1}$, $T_{x_0} \Omega \subset T_{x_0} S^{d-1}$ and hence, $x_0 \perp \tau_{i}$, for $i=1,2$. Therefore, letting $\tau_3 = x_0$, we can complete the orthonormal set ${\left\lbrace\tau_i \right\rbrace}_{i=1}^{3}$ to an orthonormal basis ${\left\lbrace \tau_i\right\rbrace}_{i=1}^{d}$. 
 
 Since $\Omega$ is a $C^{1,\beta}$ submanifold of dimension $2$, there exists a neighborhood $U_0$ of $x_0$ such that $\Omega \cap U_{0}$ can be written as a $C^{1,\beta}$ graph over $P_{x_0}$. More specifically, there exist $d-2$ $C^{1,\beta}$ functions $\psi_i$ on a neighborhood $G$ of $(0,0)$ in $\RR^{2}$ such that $\psi_1(0,0)=1$,  $\psi_{i}(0,0)=0$ for $i>1$ and: 
  \begin{equation}
 \Omega \cap U_0=\left\lbrace x_1 \tau_1 + x_2 \tau_2 + \sum_{i=1}^{d-2} \psi_{i}(x_1,x_2) \tau_{i+2}; (x_1,x_2) \in G \right\rbrace
 \end{equation}
 Denote by $\Psi: G \rightarrow {\Omega \cap U_0}$ the $C^{1,\beta}$ diffeomorphism $\Psi(x_1,x_2)= x_1 \tau_1 + x_2 \tau_2 + \sum_{i=1}^{d-2} \psi_{i}(x_1,x_2) \tau_{i+2}$.
 Let $U=U_{0} \cap S^{d-1}$, and $V=U_{1}^{\epsilon}$ (where $U_{1}^{\epsilon}$ is defined as in ($\ref{polarnhood}$)) for some $\epsilon <1$. $V$ is an open neighborhood of $x_0$ and: \begin{align}
 y \in \Sigma \cap V & \iff y= \lambda y_0 , \text{ where } y_0 \in \Omega \cap  U_{0}, { } \lambda \in (1-\epsilon, 1+\epsilon), \\ & \iff y = \lambda \Psi(x_1,x_2), \text{ where } \lambda \in (1-\epsilon, 1+ \epsilon), \text{ } (x_1,x_2) \in {\RR}^{2}. \end{align}
 
 Letting $\Phi: G \times (1-\epsilon, 1+\epsilon) \rightarrow {\RR}^d$ be defined as $\Phi((x_1,x_2),\lambda)=\lambda \Psi(x_1,x_2)$, we see that $\Phi$ is a $C^{1}$ diffeomorphism  on $G \times (1-\epsilon, 1+\epsilon)$ and :
 $$\Phi(G\times (1-\epsilon,1+\epsilon))=\Sigma \cap V. $$
Hence in the neighborhood of every non-zero point, $\Sigma$ is a $C^{1}$ manifold.
Consequently, every non-zero point of $\Sigma$ is flat.  Another application of Theorem $\ref{consPTT2}$  provides us with a $\gamma >0$ such that $\Sigma$ is a $C^{1,\gamma}$ submanifold of dimension 3 in a neighborhood of every non-zero point.

 \end{proof}
 
 We obtain the following corollary as a consequence of Theorem $\ref{onesingularity}$.
 
 \begin{corollary}\label{tangentsof3uniform}
 Let $\mu$ be a $3$-uniform  measure in $\RR^{d}$. If $x_0 \in \mbox{supp}(\mu)$, and $\mbox{Tan}(\mu,x_0)=\left\lbrace c \nu , c>0 \right\rbrace$, where $\nu$ is normalized so that $\nu(B(0,1))=\omega_{n}$, one of the following statements hold:
 
 \begin{itemize}
 \item $\nu=\mathcal{H}^{3} \res V_{x_0}$ where $V_{x_0}$ is a $3$-dimensional subspace.
 \item The support of $\nu$ is not a plane, and for all $z_0 \in \mbox{supp}(\nu)$, $z_0 \neq 0$ we have $$\mbox{Tan}(\nu,z_0)=\left\lbrace c \mathcal{H}^{3} \res V_{z_0}, c>0 \right\rbrace$$ where $V_{z_0}$ is a $3$-dimensional subspace.
 \end{itemize}
 \end{corollary}
 \begin{proof}
 By Theorem $\ref{uniqueness}$, if $x_0 \in \mbox{supp}(\mu)$, there exists a unique conical $3$-uniform measure $\nu$ such that $\mbox{Tan}(\mu,x_0)=\left\lbrace c \nu , c>0 \right\rbrace $.
But by Theorem $\ref{onesingularity}$, Theorem $\ref{uniqueness}$ and Corollary $\ref{supportrect}$, $\nu=c \mathcal{H}^{3} \res (\mbox{supp}(\nu))$ where $\mbox{supp}(\nu)$ is a $3$-dimensional subspace or a $C^{1,\alpha}$-manifold away from $0$. \end{proof}

\section{The general case: dimension reduction of singular sets}

In this section, we will use the base case to deduce the Hausdorff dimension of the singular set of any $n$-uniform measure. We first prove a theorem about the convergence of the set of singularities of a sequence of blowups. 
Once this theorem is proven, we will have all the tools we need to apply a dimension reduction argument using the base case.

Let us start with some notations.
The measure $\mu_{x ,r}$ is defined as:
\begin{equation}
\mu_{x ,r}(A)= \omega_{n} (\mu(B(x ,r)))^{-1} \mu(r A+x),
\end{equation}
for $A \subset \RR^{d}$.
In particular if $\mu$ is $n$-uniform and $z \in \mbox{supp}(\mu)$, then it follows from Theorem $\ref{uniqueness}$ that  for any sequence $\eta_j \downarrow 0$  
\begin{equation}
\mu_{z , \eta_{j}} \rightharpoonup \mu^{z},
\end{equation}
where $\mu^{z}$ is the normalized tangent measure at $z$ as defined in Definition $\ref{normalized}$.

The following fact, which is a direct consequence of the definition of the functional $F$ from Definition $\ref{functional}$, will be used often in this section: if $\Phi$ is a flat measure, then $F(\Phi)=0$ and $F(\Phi_{0,C})=0$ for any $C>0$.

Recall Definition $\ref{flat}$.
\begin{definition}
Let $\mu$ be an $n$-uniform measure in $\RR^{d}$.
If $x_0$ is a non-flat point of \mbox{supp}($\mu$), we call it a singularity of $\mu$.
We denote by $\mathcal{S}_{\mu}$ the set of singularities of $\mu$, namely:
$$\mathcal{S}_{\mu}= \left\lbrace x \in \mbox{supp}(\mu), x \mbox{ is not a flat point } \right\rbrace.$$
\end{definition}
We start with a lemma which states that under the appropriate conditions, blow-ups along the same sequence of points satisfy some sort of connectedness property. 
\begin{lemma}\label{connectednessblowups}
Let $\mu$ be an $n$-uniform measure in $\RR^{d}$, $\left\lbrace x_k \right\rbrace_{k} \subset \mbox{supp}(\mu)$ and $\left\lbrace \tau_{k} \right\rbrace_k$, $\left\lbrace \sigma_{k} \right\rbrace_{k}$ sequences of positive numbers decreasing to $0$.
We also assume that $\sigma_k < \tau_k$ and that there exist $n$-uniform measures $\alpha$ and $\beta$ such that:
$$\mu_{x_k,\tau_k} \rightharpoonup \alpha \mbox{ and } \mu_{x_k,\sigma_k} \rightharpoonup \beta.$$
Then:
$$F(\alpha) < \epsilon_{0} \implies F(\beta) < \epsilon_{0},$$
where $F$ is the functional from Definition $\ref{functional}$ and $\epsilon_0$ is the constant from $\ref{functionalflatness} $. \end{lemma}

\begin{proof}
The proof of this lemma is similar to  Preiss' proof of Theorem 2.6  in $\cite{P}$. In particular, it follows closely the proof of Theorem 6.10 in $\cite{Del}$ which is a reformulation of Preiss' theorem for uniform measures.

Assume that  $F(\alpha)<\epsilon_{0}$ and $F(\beta) \geq \epsilon_{0}$.
Then by Theorem $\ref{contfunctional}$ there exists $0<\kappa<\epsilon_{0}$ and $k_0>0$ so that for $k>k_0$ , $$ F(\mu_{x_k , \sigma_k }) > \kappa \mbox{ and } F(\mu_{x_k , \tau_k }) < \kappa.$$
For $k>0$, define the function $f_k: (0,\infty) \rightarrow (0,\infty)$ to be 
$$f_k(r)=F(\mu_{x_k , r}).$$
If $s_j$ is a sequence of positive numbers converging to some $s_0>0$, $\mu_{x_k, s_l} \rightharpoonup \mu_{x_k, s_0}$. Therefore $f_k$
is continuous in $r$ away from $0$, for all $k>0$. 
So for every $k>k_0$, there exists $\delta_{k} \in [\sigma_{k} , \tau_{k} ]$ so that:
\begin{equation} \label{deltak}
F(\mu_{x_k , \delta_{k}}) = \kappa \mbox{ and } F(\mu_{x_k , r}) \leq \kappa \mbox{ for } r \in [\delta_k, \tau_k].
\end{equation}
By Theorem $\ref{weakconv}$, without loss of generality, by passing to a subsequence, $$\mu_{x_k , \delta_{k}} \rightharpoonup \xi $$
where $\xi$ is a Radon measure. 
We claim that $\xi$ is $n$-uniform and $\xi(B(0,1))=\omega_{n}$. Pick $y \in \mbox{supp}(\xi)$ and $R>0$.
First note that since the $\mu_{x_k , \delta_{k}}$ are $n$-uniform and all have the same constant $\omega_n$ (being normalized), we can apply Theorem $\ref{weakconv2}$ to obtain a sequence $y_{k_j}$ of points in $\mbox{supp}(\mu_{x_{k_j} , \delta_{k_j}})$ such that $y_{k_j} \to y$. Without loss of generality, by passing to a subsequence, $y_j \to y$.
Fix $\epsilon>0$. There exists $j_0$ such that: 
$$ j>j_0 \implies |y-y_j| < \frac{\epsilon}{4}.$$
On one hand, we have:
\begin{align*}
\xi(B(y,R)) & \leq \liminf \mu_{x_j , \delta_{j}}(B(y,R)), \\ & \leq \liminf \mu_{x_j , \delta_{j}}(B(y_j,R+\frac{\epsilon}{4})), \\ &=\omega_{n} \left(R+\frac{\epsilon}{4}\right)^{n}.
\end{align*}
On the other hand
\begin{align*}
\xi(B(y,R)) & \geq \limsup \mu_{x_j , \delta_{j}}(B(y,R-\frac{\epsilon}{8})), \\ & \geq \limsup \mu_{x_j , \delta_{j}}(B(y_j,R-\frac{3\epsilon}{8})), \\ &=\omega_{n} \left(R-\frac{3\epsilon}{8}\right)^{n}.
\end{align*}
Hence, for $y \in \mbox{supp}(\xi)$ and $R>0$
$$ \omega_{n} \left(R-\frac{3\epsilon}{8}\right)^{n} \leq \xi(B(y,R)) \leq \omega_{n} \left(R+\frac{\epsilon}{4}\right)^{n}.$$
Since $\epsilon$ was chosen arbitrarily, $$\xi(B(y,R))=\omega_{n} R^{n},$$ thus proving that $\xi$ is $n$-uniform.

By Theorem $\ref{contfunctional}$, $F(\xi)=\kappa$. In particular $\xi$ is not flat.
We now show that our assumptions imply that $\xi$ is flat at infinity. By Theorem $\ref{flatnessinfty}$ this is a contradiction with $\xi$ not being flat.

We first claim that $\frac{\tau_{k}}{\delta_{k}} \to \infty$.
Assume that $\frac{\tau_{k}}{\delta_{k}} \to C$, $C\geq 1$.
Letting $\beta_k = \frac{\tau_{k}}{\delta_{k}}$ and writing $ \mu_{x_k , \tau_{k}} = {\beta_{k}}^{-n} T_{0, \beta_{k}}[\mu_{x_k, \delta_{k}}]$ , we have $$\mu_{x_k , \tau_{k}} \rightharpoonup \xi_{0,C}$$ since $C \neq 0$, $\mu_{x_k, \delta_{k}} \rightharpoonup \xi$ and $\xi(B(0,C))=\omega_{n} C^{n}$.
But $\mu_{x_k , \tau_{k}} \rightharpoonup \alpha$ hence $\alpha=\xi_{0,C}$. The fact that  $\alpha$ is flat and $\xi$ is not would yield a contradiction.

Now fix $R>1$. Since $\frac{\tau_{k}}{\delta_{k}} \to \infty$, there exists $k_1>k_0$ such that for $k>k_1$, we have $$R \delta_{k} \in [\delta_{k} , \tau_{k}].$$
In particular, since $k_1>k_0$, if $k>k_1$ we also have, by $\eqref{deltak}$, $F(\mu_{x_k , R \delta_k}) \leq \kappa$. We deduce that:
$$\limsup_{k} F(\mu_{x_k , R \delta_k} ) \leq \kappa.$$ 
Since $F_{s}(\mu_{x_k , R \delta_k} , \xi_{0,R}) = R^{-n-1} F_{Rs}(\mu_{x_k , \delta_k}, \xi)$ for every $s>0$, $\lim_{k \to \infty} F_{s}(\mu_{x_k , R \delta_k}, \xi_{O,R})=0$ for every $s>0$ and hence $\mu_{x_k , R \delta_k} \rightharpoonup \xi_{0,R}$. Consequently, by Theorem $\ref{contfunctional}$:
\begin{equation}\label{infinityxi}
 F( \xi_{0,R}) \leq \kappa.
\end{equation}
Choosing  $R_l \uparrow \infty$, we have $ \xi_{O,R_l} \rightharpoonup \psi $ where $\psi$ is the normalized tangent of $\xi$ at $\infty$ by Theorem $\ref{uniqueness}$. Therefore, by $\eqref{infinityxi}$, $F(\psi) \leq \kappa < \epsilon_{0}$.
But by Theorem $\ref{functionalflatness}$ this implies  that $\xi$ is flat which contradicts $F(\xi) = \kappa$. 
\end{proof}
\begin{remark}
\begin{enumerate}
\item I would like to thank Max Engelstein for discussions about the connectedness of cones of pseudo-tangent measures along a subsequence which brought this argument to my attention.
It is worth noting that cones of pseudo-tangent sets (the same holds for measures) are $\textit{not}$ connected in general. A counter-example can be found in Remark 5.5 in $\cite{BL}$. 
 The interested reader can refer to $\cite{P}$, $\cite{KPT}$ and $\cite{BL}$ for more detailed discussions of cones of tangent measures.
 
 \item  Lemma $\ref{connectednessblowups}$ would imply that pseudo-tangent measures are in fact connected along the same subsequence if not for the condition $\sigma_k < \tau_k$. This apparently technical condition turns out to be necessary for this proof to work. Whether this is just a feature of this specific proof or an actual necessary condition is not clear to the author and seems like an interesting question in its own right.
\end{enumerate}

\end{remark}
We can now prove a useful theorem about the behavior of singular sets under blow-ups.
\begin{theorem} \label{accumsing}
 Let $\mu$ be an $n$-uniform measure in $\RR^d$, $x_0 \in \mbox{supp}(\mu)$, $\left\lbrace x_j \right\rbrace _{j} \subset \mathcal{S}_{\mu}$, $\left\lbrace r_j \right\rbrace _{j}$ any sequence of positive numbers decreasing to $0$.
Also assume that $y_j=\frac{x_j - x_0}{r_j} \in \overline{B(0,1)}$, $y_j \to y$.
Then $$y \in \mathcal{S}_{\mu^{x_0}},$$
where $\mu^{x_0}$ is the normalized tangent at $x_0$ as defined in Definition $\ref{normalized}$.
\end{theorem}
\begin{proof}
Without loss of generality, $x_0 = 0$. Denote $\mu^{x_0}$ by $\nu$, and $\mu^{x_j}$ by $\nu_j$ where $\mu^{x_j}$ are the normalized tangents at $x_j$.
Let us start with some remarks.

We first claim that there exists a conical, non-flat $n$-uniform measure $\nu^{\infty}$ such that a subsequence of $\left\lbrace \nu_{j} \right\rbrace_j$ converges weakly to $\nu^{\infty}$.
Since $x_j \in \mathcal{S}_{\mu}$, $\nu^{j}$ is non-flat for every $j>0$. The fact that $\nu_{j}$ is conical and Theorem $\ref{functionalflatness}$ then imply that  $F(\nu_j) > \epsilon_{0}$ for all $j>0$ where $F$ is the functional defined in Definition $\ref{functional}$.
Moreover, since $\sup_{j}(\nu_j(B(0,R)) )= \omega_{n} R^{n}< \infty$, there exists a Radon measure $\nu^{\infty}$ and a subsequence of $\nu_j$ converging to $\nu^{\infty}$. Without loss of generality, $\nu_j \rightharpoonup \nu^{\infty}$.
Moreover, $\nu^{\infty}$ is $n$-uniform. The proof of this fact is exactly the same as the proof of the fact that $\xi$ in Lemma $\ref{connectednessblowups}$ is $n$-uniform.

By Theorem $\ref{contfunctional}$, $F(\nu_{j}) \to F(\nu^{\infty})$ and hence $F(\nu^{\infty}) \geq \epsilon_{0}$.
Moreover, since each $\nu_{j}$ is conical and $\nu_{j} \rightharpoonup \nu^{\infty}$, it follows that for any $r>0$:
\begin{align*}
T_{0,r}[ \nu^{\infty}] &= \lim T_{0,r}[\nu_{j}], \\
&=r^{n} \lim \nu_{j}, \\
&=r^{n} \nu^{\infty}.
\end{align*}
This proves that $\nu^{\infty}$ is conical.

We also claim that 
\begin{equation} \label{nuy}
y \in \mbox{supp}(\nu) \mbox{ and } \mu_{x_j , r_j} \rightharpoonup \nu_{y}. 
\end{equation} where  $\nu_{z}$ denotes $T_{z,1}[\nu]$ whenever $z \in \mbox{supp}(\nu)$. 
 Indeed, let $\delta > 0$.
Then:
\begin{align*}
\nu(B(y,\delta)) & \geq \limsup_{i \to \infty} \mu_{0,r_i}\left(B\left(y,\frac{\delta}{4}\right)\right) \\ & =\limsup_{i \to \infty} \omega_{n}(\mu(B(0,r_i)))^{-1} \mu\left(B\left(r_i y, \frac{r_i \delta}{4}\right)\right).
\end{align*}
But for $i$ large enough $|y-y_i| \leq \frac{\delta}{8}$ implying that $B(x_{i}, r_i \frac{\delta}{8}) \subset B(r_i y, \frac{r_i \delta}{4})$. Consequently, $$\nu(B(y,\delta)) >0$$ since $$(\mu(B(0,r_i)))^{-1} \mu\left(B(x_{i}, r_i \frac{\delta}{8})\right) = \frac{\delta^{n}}{8^n}.$$

Let us prove the second part of $\eqref{nuy}$.
Recall Definition $\ref{L(r)}$.

Fix $R>0$. Let $\phi \in \mathcal{L}(R)$. Then, on one hand, for $j$ large enough that $|y_j| \leq 2$, we have:
\begin{align}
\left| \int \phi(z) d\mu_{x_j , r_j}(z) - \int \phi(z) dT_{y_j ,1}[\nu](z) \right| 
& =\left| \int \phi( z-y_j) d\mu_{0,r_j}(z) - \int \phi(z-y_j) d\nu(z) \right|, \nonumber \\
&\leq F_{R+2} (\mu_{0,r_j}, \nu) , 
\end{align}
since $\phi_j(z)=\phi(z-y_j) \in \mathcal{L}(R+2)$ .
On the other hand,
\begin{align}
\left| \int \phi(z) dT_{y_j ,1}[\nu](z)- \int \phi(z) dT_{y,1}[\nu](z) \right| &= \left| \int \left(\phi(z-y_j) - \phi(z-y)\right) d\nu(z) \right| , \nonumber \\
&\leq |y-y_j| \nu(B(0,R+2)),
\end{align}
since $Lip(\phi) \leq 1$, $\phi_{j}$ and $\phi_{y}$ are supported in $B(0,R+2)$ where we define $\phi_{y}(z)=\phi(z-y)$.
This gives, taking the supremum over all $\phi \in \mathcal{L}(R)$:
$$ F_R(\mu_{x_j , r_j} , \nu_{y} ) \leq F_{R+2} (\mu_{0,r_j}, \nu)+ |y-y_j| \nu(B(0,R+2)), $$
for $j$ large enough.
Letting $j \to \infty$, we get $\eqref{nuy}$ since $R$ was chosen arbitrarily. 

Our proof will now go as follows: we construct sequences of positive numbers $\sigma_{k}$ and $\tau_{k}$ decreasing to $0$ such that $\mu_{\tilde{x}_k, \sigma_{k}}$ converges weakly to $\nu^{\infty}$ and $\mu_{\tilde{x}_k, \tau_{k}}$ converges weakly to $\alpha$ the normalized tangent measure to $\nu$ at $y$. Here, $\tilde{x}_k$ is a subsequence of $x_k$.
We then use Lemma $\ref{connectednessblowups}$ to deduce that $\alpha$ cannot be flat.

Let us first construct a decreasing sequence $\tilde{\sigma}_j$ such that 
\begin{equation}\label{sigma}
\frac{\tilde{\sigma}_{j}}{r_j} \to 0 \mbox{ and } \mu_{x_j, \tilde{\sigma}_{j}} \rightharpoonup \nu^{\infty}.
\end{equation}
Let $t_{j}=\frac{1}{j}$.
By Theorem $\ref{uniqueness}$, the blow-ups at a point  converge to the tangent along $\textbf{\textit{any}}$ sequence going to 0. Moreover this tangent is unique up to normalization. Thus, for every $k$, we have
$$\mathcal{F}(\mu_{x_k, \frac{1}{j}}, \nu_{k}) \to 0, $$
where $\mathcal{F}$ is the metric on Radon measures from Definition $\ref{L(r)}$.
Now construct inductively a decreasing sequence $\left\lbrace l_{k} \right\rbrace _{k}$
such that, for all $k>0$ 
\begin{equation} \label{convnuinfty}
 l \geq l_k \implies t_{l} < {r_{k}}^{2} \mbox{ and } \mathcal{F}(\mu_{x_{k}, t_{l}} , \nu_{k} ) < \frac{1}{2^{k}}.
\end{equation} 
Let $\tilde{\sigma}_{j} = t_{l_j}$ and $\rho_{j}= \frac{\tilde{\sigma}_{j}}{r_j}$.

We remark that since $\rho_j \downarrow 0$, \begin{equation}\label{alpha2}
\left( \nu_{y} \right) _{0, \rho_{j}} \rightharpoonup \alpha 
\end{equation}  
where $\alpha $ is the normalized tangent measure to $\nu_y$ at $0$. Equivalently, this is the normalized tangent measure to $\nu$ at $y$. Indeed, since $\nu_{y}=T_{y,1}[\nu]$ and $T_{0,\rho_{j}}\circ T_{y,1}= T_{y,\rho_j}$, we have \begin{align*}
{\rho_{j}}^{-n}T_{0,\rho_j}[\nu_{y}] &= {\rho_{j}}^{-n}T_{0,\rho_j}[ T_{y,1} [\nu]] , \\
&={\rho_{j}}^{-n} T_{y,\rho_j}[\nu].
\end{align*}

We now construct a sequence $\tilde{\tau}_{k}$ such that:
$$\mu_{{x}_{l_k} , \tilde{\tau}_k} \rightharpoonup \alpha,$$
for some subsequence $x_{l_k}$ of $x_k$.

For every $k$ there exists $l_k>k$, $l_k > l_{k-1}$ such that whenever $l>l_k$ 
\begin{equation}\label{tau}
F_{1} (\mu_{x_l , r_l} , \nu_{y}) < \frac{1}{k} {\rho}_{k}^{n+1} \mbox{ and } \rho_{l} < \rho_{k},
\end{equation} 
since $\mu_{x_l , r_l} \rightharpoonup \nu_{y}$ and $\rho_{k} \to 0$.
Let $\tilde{\tau_{k}} = r_{l_k} \rho_{k}$ and $\tilde{x}_k= x_{l_k}$.

We claim that \begin{equation}\label{alpha}
\mu_{\tilde{x}_k , \tilde{\tau}_k} \rightharpoonup \alpha.
\end{equation}
Indeed, fix $R>0$.
On one hand, for $k$ large enough that $R\rho_{k} \leq 1$ \begin{align*}
F_{R}(\mu_{\tilde{x}_k , \tilde{\tau}_k} , \rho_{k}^{-n} T_{0,\rho_{k}} [\nu_{y}]) &= F_{R}(\rho_{k}^{-n} T_{0,\rho_k}[\mu_{x_{l_k}, r_{l_k}}] , \rho_{k}^{-n} T_{0,\rho_{k}}[\nu_{y}]), \\
& =\rho_{k}^{-n-1} F_{R \rho_{k}}(\mu_{x_{l_k}, r_{l_k}} , \nu_{y}) , \\ & \leq \rho_{k}^{-n-1} F_{1}(\mu_{x_{l_k}, r_{l_k}}, \nu_{y}), \\ &< \frac{1}{k}.
\end{align*}
The laws of composition used in this calculation are explained in Lemma 2.4 of $\cite{B}$.  

On the other hand, $F_R (\rho_{k}^{-n} T_{0,\rho_{k}} [\nu_{y}], \alpha) \to 0$ by $\eqref{alpha2}$.
Since $$F_R(\mu_{\tilde{x}_k , \tilde{\tau}_k},\alpha) \leq F_{R}(\mu_{\tilde{x}_k , \tilde{\tau}_k} , \rho_{k}^{-n} T_{0,\rho_{k}} [\nu_{y}]) + F_R (\rho_{k}^{-n} T_{0,\rho_{k}} [\nu_{y}], \alpha),$$ $F_R(\mu_{\tilde{x}_k , \tilde{\tau}_k},\alpha) \to 0$. 
This proves $\eqref{alpha}$.

Rename $\tilde{x}_k$, $\tilde{\sigma}_{l_k}$ and $\tilde{\tau}_{k}$ to be $x_k$, $\sigma_k$ and $\tau_k$.
We have proven that: 
$$ \mu_{x_k, \sigma_{k}} \rightharpoonup \nu^{\infty} \mbox{ and } \mu_{x_k, \tau_k} \rightharpoonup \alpha, $$
with $\sigma_{k}< \tau_{k}$ ( since $\rho_{l_k} < \rho_{k}$ by $\eqref{tau}$) and $\nu^{\infty}$ conical and non-flat. 

If $\alpha$ were flat,
we would have $F(\alpha)=0<\epsilon_{0}$ and $F(\nu^{\infty}) \geq \epsilon_{0}$. This contradicts Lemma $\ref{connectednessblowups}$.
Therefore $\alpha$ cannot be flat and $y \in \mathcal{S}_{\nu}$.

\end{proof}
\begin{remark}
The proof of $\eqref{alpha}$ is similar to the proof of Lemma 2.6 in $\cite{B}$. 
\end{remark}

We now use Theorem $\ref{connectednessblowups}$ to deduce two important corollaries.

\begin{corollary}\label{isolate3unif}
Let $\mu$ be a $3$-uniform measure in $\RR^{d}$. Then the singular set of $\mu$ is discrete. Namely, for every $K$ compact subset of $\RR^{d}$, $\left| \mathcal{S}_{\mu} \cap K \right| < \infty$. Here, $|A|$ denotes the cardinality of the set $A \subset \RR^{d}$.
\end{corollary}
\begin{proof}
Assume not. Then there exists $K$ compact subset of $\RR^{d}$ such that $\left|\mathcal{S}_{\mu} \cap K\right| = \infty$ . In particular there exists a sequence of points $\left\lbrace x_j \right\rbrace _{j} \subset \mathcal{S}_{\mu} \cap K$ converging to some  $x_{\infty} \in K$. Moreover, $x_{\infty} \in \mbox{supp}(\mu)$ since the support of a measure is a closed set.
Let $r_j = |x_j-x_{\infty}|$ and $y_j = \frac{x_{j}-x_{\infty}}{r_j}$. Then by Theorem $\ref{uniqueness}$, $\mu_{x_{\infty} , r_j } \rightharpoonup \nu$, $\nu$ normalized tangent to $\mu$ at $x_{\infty}$ and by compactness, we can assume by passing to a subsequence if necessary that $y_j \to y \in \partial B(0,1)$. 
By $\eqref{nuy}$, $y \in \mbox{supp}(\nu)$. Since $y \neq 0$, $y$ must be a flat point of $\mbox{supp}(\nu)$ by Corollary $\ref{tangentsof3uniform}$. This contradicts Theorem $\ref{accumsing}$.

\end{proof}

\begin{corollary} \label{convsingset}
Let $\mu$ be an $n$-uniform measure in $\RR^{d}$, $x_0 \in \mbox{supp}(\mu)$, $\nu$ normalized tangent to $\mu$ at $x_0$, $\left\lbrace r_j \right\rbrace _{j}$ sequence of positive radii decreasing to $0$, $\epsilon >0$.
Then there exists $N = N(\epsilon)$ such that:
\begin{equation}
n \geq N \implies \frac{\mathcal{S}_{\mu} - x_0}{r_j} \cap \overline{B(0,1)} \subset \left( \mathcal{S}_{\nu} \right) _{\epsilon} 
\end{equation}
where $\left( \mathcal{S}_{\nu} \right) _{\epsilon} = \left\lbrace y \in \RR^{d} ; dist(y, \mathcal{S}_{\nu}) < \epsilon \right\rbrace$.
\end{corollary}
\begin{proof}
Suppose not. Then we can construct a subsequence $\left\lbrace s_j \right\rbrace$ of $\left\lbrace r_j \right\rbrace$ so that: 
$$ \frac{\mathcal{S}_{\mu} - x_0}{s_j} \cap \overline{B(0,1)} \cap \left( \RR^{d} \backslash { \left( \mathcal{S}_{\nu} \right) _{\epsilon}}\right) \neq \emptyset .$$
Consequently, we can find points $x_j \in \mathcal{S}_{\mu}$ such that $y_{j} = \frac{x_j - x_0}{s_j} \in \overline{B(0,1)} $ , $dist(y_j , \mathcal{S}_{\nu}) \geq \epsilon $ and $y_j \to y$. In particular, $y \notin \mathcal{S}_{\nu}$ since $d(y , \mathcal{S}_{\nu}) \geq \epsilon$. This contradicts Theorem $\ref{accumsing}$.
\end{proof}
We are now ready to prove the main theorem of this section.

 \begin{theorem}
 Let $\mu$ be an $n$-uniform measure in $\RR^{d}$ , $ 3 \leq n \leq d$. Then 
 \begin{equation} \label{hdimsing}
 dim_{\mathcal{H}}(\mathcal{S}_{\mu}) \leq n-3,
\end{equation}
where $dim_{\mathcal{H}}$ denotes the Hausdorff dimension.
\end{theorem}
\begin{proof}

The theorem holds for $n=3$ by Theorem $\ref{onesingularity}$.
Let $m< d$ and assume the theorem holds for all $l$-uniform measures in $\RR^{d}$ such that $l < m$. We want to prove that it holds for $m$-uniform measures.

Let $\mu$ be an $m$-uniform measure.
Suppose that $s \in \RR_{+}$ is such that $\mathcal{H}^{s}(\mathcal{S}_{\mu}) >0$.

We start with an overview of the proof. We first prove that there exists a singular point $x_0$ of the support of $\mu$ and  a tangent $\nu$ of $\mu$ at $x_0$ such that the following holds
$$\mathcal{H}^{s}(\mathcal{S}_{\nu} \cap \overline{B(0,1)})>0.$$
From that, we deduce that there exists some non-zero singular point $\xi$ of the support of $\nu$ such that the tangent $\lambda$ to $\nu$ at $\xi$ satisfies $\mathcal{H}^{s}(\mathcal{S}_{\lambda} \cap \overline{B(0,1)})>0$. Note that by Theorem $\ref{uniqueness}$, since $\mu$ is $m$-uniform, $\nu$ is conical. The advantage of repeating this procedure is that since $\nu$ is conical, the support of $\lambda$ is in fact translation invariant along the vector $\xi$ and so $\lambda$ can be decomposed into $\mathcal{L}^1 \times \lambda_0$ where $\mathcal{L}^1$ is $1$-Lebesgue measure and $\lambda_0$ is $(m-1)$-uniform. We then apply the induction hypothesis to $\lambda_0$ to finish the proof.

We find  a singular point $x_0$ of the support of $\mu$ such that the following holds. Let $\nu$ be the normalized tangent to $\mu$ at $x_0$. Then:
$$\mathcal{H}^{s}(\mathcal{S}_{\nu} \cap \overline{B(0,1)})>0.$$

By Lemma 4.6 in $\cite{M}$, $$\mathcal{H}^{s}(\mathcal{S}_{\mu}) >0 \iff \mathcal{H}^{s}_{\infty}(\mathcal{S}_{\mu}) >0.$$
We will use $\mathcal{H}^{s}_{\infty}$ instead of $\mathcal{H}^{s}$ to allow a larger choice of coverings.
Since $\mathcal{H}_{\infty}^{s}(\mathcal{S}_{\mu} )>0$, there exists a compact set $K$ such that $\mathcal{H}_{\infty}^{s}(\mathcal{S}_{\mu} \cap K )>0$. Let $\tilde{\mathcal{S}_{\mu}}= \mathcal{S_{\mu}} \cap K$.
We have
\begin{equation}\label{densityupper'}
\theta^{s,*}(\mathcal{H}^{s}_{\infty}\res \tilde{{\mathcal{S}_{\mu}} }, z) \geq 2^{-s},
\end{equation}
for $\mathcal{H}^{s}$-almost every $z \in  \tilde{\mathcal{S}_{\mu}}$.
This follows from Theorem 3.26 (2), in $\cite{S}$ since $\tilde{\mathcal{S}_{\mu}}$ is a compact subset of $\RR^d$.
In particular, there exists $x_0 \in \tilde{\mathcal{S}_{\mu}}$ such that:
\begin{equation}
\theta^{s,*}(\mathcal{H}^{s}_{\infty}\res \tilde{{\mathcal{S}_{\mu}} }, x_0) \geq 2^{-s},
\end{equation}

Consequently, there exists a sequence of  radii $\left\lbrace r_{j} \right\rbrace _{j}$ decreasing to $0$ such that:
$$\mathcal{H}_{\infty}^{s} \left(\overline{B(0,1)} \cap \frac{\tilde{{\mathcal{S}}_{\mu}} - x_0}{r_j}\right) \geq 2^{-s}. $$
Since $r_j \downarrow 0$, $\mu_{ x_0 , r_j } \rightharpoonup \nu$ where $\nu= \mu^{x_0}$, the normalized tangent to $\mu$ at $x_0$.
By Theorem $\ref{convsingset}$, for all $\epsilon >0$, there exists $j_{0}$ such that: 
\begin{equation} \label{nhoodsing}
\frac{\tilde{\mathcal{S}_{\mu}} - x_{0}}{r_j} \cap \overline{B(0,1)} \subset \left( \mathcal{S}_{\nu} \right)_{\epsilon}  \mbox{ whenever }         j \geq j_{0}.
\end{equation}
Pick $\delta > 0$ and let $\left\lbrace E_k \right\rbrace _{k}$ be a covering of $\tilde{\mathcal{S}_{\nu}} = \mathcal{S}_{\nu} \cap \overline{B(0,1)}$ such that:
$$\mathcal{H}^{s}_{\infty}(\tilde{\mathcal{S}_{\nu}} ) >  \omega_{s} 2^{-s} \sum_{k=1}^{\infty} (diam(E_k))^{s}  - \delta.$$
We can assume that the sets $E_k$ are open (see Theorem 4.4 in $\cite{M}$).
Since $\bigcup E_{k}$ is open, $\tilde{\mathcal{S}_{\nu}}$ is compact and $\tilde{\mathcal{S}_{\nu}} \subset \bigcup E_{k}$, we can cover $\tilde{\mathcal{S}_{\nu}}$ with finitely many $E_k$, $k=1,\ldots,K$. Letting $E$ be the union of this finite cover and $\epsilon$ be a number smaller than the minimum of the diameters of the  $E_k$'s in this finite cover, we have:
$$(\mathcal{S}_{\nu} )_{\epsilon} \subset E.$$
It follows from $\eqref{nhoodsing}$ that for $j$ large enough, we have
$$\mathcal{S}_{j} \subset E,$$
where $\mathcal{S}_{j} = \frac{\tilde{\mathcal{S}_{\mu}} - x_0}{r_j} \cap \overline{B(0,1)}$.
Hence, for $j$ large, since $\left\lbrace E_k \right\rbrace_{k=1}^{K}$ covers $\mathcal{S}_{j}$
\begin{align*}
\mathcal{H}^{s}_{\infty}(\mathcal{S}_{j} )  & \leq \omega_s 2^{-s} \sum_{k=1}^{K} (diam(E_k))^s, \\ &\leq  \mathcal{H}^{s}_{\infty}(\tilde{\mathcal{S}_{\nu}})+\delta.
\end{align*}
Since $\delta$ was chosen arbitrarily, we get $
\mathcal{H}^{s}_{\infty}(\mathcal{S}_{j} ) \leq \mathcal{H}_{\infty}^{s}(\tilde{\mathcal{S}_{\nu}})$.
Letting $j \to \infty$, we get:
$$2^{-s} \leq \limsup \mathcal{H}^{s}_{\infty}(\mathcal{S}_{j} ) \leq  \mathcal{H}_{\infty}^{s}(\tilde{\mathcal{S}_{\nu}}).$$
This gives $\mathcal{H}_{\infty}^{s}(\mathcal{S}_{\nu}) \geq  \mathcal{H}_{\infty}^{s}(\tilde{\mathcal{S}_{\nu}}) >0$. The claim is thus proved. 

The advantage of $\nu$ over $\mu$ is that it is conical. Thus if we blow up at a non-zero point of $\nu$, we claim that we obtain a measure that is translation invariant. 
Since $\mathcal{H}^{s}(\mathcal{S}_{\nu} \cap \overline{B(0,1)})>0$, by the same reasoning as for $\mu$, there exists $\xi \neq 0$, $\xi \in \mathcal{S}_{\nu} \cap \overline{B(0,1)}$ such that : $$\theta^{s,*}(\mathcal{H}_{\infty}^{s} \res \mathcal{S}_{\nu}, \xi) \geq 2^{-s}.$$ In particular, there exists a decreasing sequence $\left\lbrace s_j \right\rbrace$ such that $\mathcal{H}^{s}_{\infty} (\mathcal{S}_{\nu} \cap \overline{B(\xi, s_j)}) \geq 2^{-s} s_{j}^{s}$ and $\nu_{\xi, s_j} \rightharpoonup \lambda$, where $\lambda= \nu^{\xi}$ is the normalized tangent measure to $\nu$ at $\xi$. Since $\nu$ is uniform and $\xi$ is a singular point, $\lambda$ is a non-flat conical measure.
The same procedure as above gives:
\begin{equation}\label{dimred1}
\mathcal{H}^{s}(\mathcal{S}_{\lambda} \cap \overline{B(0,1)})>0.
\end{equation}
Let $\Sigma= \mbox{supp}(\lambda)$. We claim that 
$$ \Sigma = \RR  e_1\oplus A $$ for some $A$ subset of a $(d-1)$-plane of $\RR^d$  such that $\mathcal{H}^{m-1} \res A$ is $(m-1)$-uniform.
We will first prove that \begin{equation} \label{transinvmeasure}
T_{t \xi, 1} [\lambda]=\lambda 
\end{equation}
 for any $t>0$. 

Take $t>0$.
Then, on one hand,  noting that for $z \in \RR^{d}$:
\begin{align*}
T_{(1+t)\xi , s_j}(z) &= \frac{z-\xi - t \xi}{s_j}, \\
&= \frac{\frac{z}{1+t}- \xi}{\frac{s_j}{1+t}} , \\
&= T_{\xi, \frac{s_j}{1+t}} \circ T_{0, 1+t} (z), 
\end{align*}
we get
 \begin{align}\label{nualpha}
\nu_{(1+t)\xi, s_j} &=s_{j}^{-m} T_{\xi, \frac{s_j}{1+t}}[T_{0, 1+t}[\nu]],   \nonumber   \\
&= s_{j}^{-m} (1+t)^{m} T_{\xi,\frac{s_j}{1+t}}[\nu]  , \mbox{ since } \nu \mbox{ is conical} \nonumber \\
& \rightharpoonup \lambda,
\end{align}
since the sequence $\frac{s_j}{1+t} \to 0$ and $s_{j}^{-m} (1+t)^{m} T_{\xi,\frac{s_j}{1+t}}[\nu](B(0,1))=\lambda(B(0,1))=\omega_{m}$.

On the other hand, we have
\begin{align*}
T_{(1+t)\xi, s_j}(z)& = \frac{z-(1+t)\xi}{s_j}, \\
&=\frac{z- (1+(1-s_j)t)\xi}{s_j} -t \xi, \\
&= T_{t \xi, 1} \circ T_{(1+(1-s_j)t)\xi, s_j} (z).
\end{align*}
We now prove that 
\begin{equation}\label{transinvstep}
{s_j}^{-m} T_{(1+(1-s_j)t) \xi , s_j}[\nu] \to \lambda.
\end{equation}

Let $\phi \in \mathcal{L}(R)$.
Then, for $j$ large enough so that $|1-s_j| \leq 2$ we have:
\begin{align*}
&{s_j}^{-m} \left| \int \phi(z) dT_{(1+(1-s_j) t)\xi, s_j}[\nu](z) - \int \phi(z) dT_{(1+t)\xi, s_j}[\nu](z) \right| \\& \leq 
{s_j}^{-m}  \left| \int \left( \phi(z-(1+(1-s_j) t)\xi)  - \phi(z-(1+t)\xi) \right) dT_{0, s_j}[\nu](z) \right|, \\
& \leq {s_j}^{-m} \int_{B(0, R+(1+2|t|)|\xi|)} |s_j||\xi| |t| dT_{0,s_j}[\nu](z) , \\
& \leq |s_j||\xi| |t|\omega_{m} (R+(1+2|t|)|\xi|)^{m}.
\end{align*}
Taking the supremum over all $\phi \in \mathcal{L}(R)$, we get: 
\begin{align}\label{transinvstep'}
A_{j} &:= F_{R}({s_j}^{-m} T_{(1+(1-s_j)t) \xi , s_j}[\nu], {s_j}^{-m} T_{(1+t) \xi , s_j}[\nu]), \nonumber \\ & \leq |s_j||\xi| |t|\omega_{m} (R+(1+2|t|)|\xi|)^{m},
\end{align}
which goes to $0$ as $j \to \infty$ since $s_{j} \to 0$.
We have 
\begin{equation}
\label{triineq} F_{R}({s_j}^{-m} T_{(1+(1-s_j)t) \xi , s_j}[\nu], \lambda) \leq A_{j} + F_{R}({s_j}^{-m} T_{(1+t) \xi , s_j}[\nu], \lambda).
\end{equation}
Since $A_j \to 0$ by $\eqref{transinvstep'}$ and, according to $\eqref{nualpha}$, $F_{R}({s_j}^{-m} T_{(1+t) \xi , s_j}[\nu], \lambda) \to 0$, by using $\eqref{triineq}$, we prove $\eqref{transinvstep}$.

This proves $\eqref{transinvmeasure}$ from which it follows that
\begin{equation} \label{transinv}
\Sigma - t \xi = \Sigma \mbox{ for } t>0.
\end{equation}
Indeed, for $t>0$,
\begin{align*}
z \in \Sigma & \iff \mbox{For all } r>0, \lambda(B(z,r))>0, \\
&\iff \mbox{ For all } r>0, T_{t\xi, 1} [\lambda](B(z,r)) >0, \\ & \iff \mbox{ For all } r>0, \lambda(B(z+t\xi, r))>0, \\ &\iff z \in \Sigma - t \xi .
\end{align*}
Adding $t\xi$ on both sides of $\ref{transinv}$, we see that
\begin{equation}\label{transinvfinal}
\Sigma - t \xi = \Sigma \mbox{ for } t \in \RR.
\end{equation}

Let $e_1 = \frac{\xi}{|\xi|}$ and $A= \left\lbrace x \in \Sigma ; x.{e_1}=0 \right\rbrace $.  
We claim that 
\begin{equation}\label{directproduct} \Sigma= \RR e_1 \oplus A.  
\end{equation}
On one hand, if $z \in \RR e_1 \oplus A$, then there exists $z' \in A$ and $t \in \RR$ such that:
$$z= z' + te_1.$$
Since $A\subset \Sigma$ by definition, this implies that $z \in \Sigma+ te_1$ and consequently, $z \in \Sigma$ by $\eqref{transinv}$.
On the other hand, if $z \in \Sigma$, we can write:
$$ z= (z - \left\langle z,e_1 \right\rangle e_1)+\left\langle z,e_1 \right\rangle e_1.$$
Let $t_1 =  \left\langle z,e_1 \right\rangle$. By $\eqref{transinv}$, $z-t_1e_1 \in \Sigma$. Moreover, $\left\langle z-t_1e_1 , e_1 \right\rangle =0$. Therefore, $z-t_1e_1\in A$ and $z \in \RR e_1 + A$.
The uniqueness of such a decomposition follows from the fact that $\RR e_1$ and $A$ are orthogonal by construction.
This proves $\eqref{directproduct}$.

So there exists $c>0$ so that $\lambda = c {\omega_{m}}^{-1} \mathcal{H}^{m} \res \left( \RR e_1 \oplus A \right)$ by Corollary $\ref{supportrect}$.
By Theorem 3.11 in $\cite{KoP}$, $\lambda_{0}= \mathcal{H}^{m-1} \res A$ is $(m-1)$-uniform.

The final step consists in proving that \begin{equation}\label{almostproduct}\mathcal{S}_{\lambda} \subset \RR e_1 \oplus \mathcal{S}_{\lambda_{0}} \cong \RR \times  \mathcal{S}_{\lambda_{0}}.
\end{equation}

We start by proving that if $y \in A $ is a $(m-1)$-flat point of $\lambda_0$, and $t \in \RR$, then  $te_1+y \in \Sigma$ is an $m$-flat point of $\lambda$.
By Theorem $\ref{PTT2}$, if $y$ is a flat point of $\lambda_0$, since $\lambda_0$ is an $(m-1)$-uniform measure, there exists a neighborhood $U'$ of $y$ in $\RR^{d-1}$ (here $\RR^{d-1}$ is identified with the set $\left\lbrace z \in \RR^{d} ; \left\langle z  , e_1 \right\rangle = 0 \right\rbrace$) such that $A \cap U'$ is a $C^{1}$ manifold.
More precisely, there exists $(d-m+1)$ $C^1$- diffeomorphisms $\left\lbrace\psi_j \right\rbrace_{j}$  from a neighborhood $G$ of $\RR^{m-1}$ to $\RR$ such that:
\begin{equation}\label{m-1-manif}
U' \cap A = \left\lbrace z_2 e_2 + \ldots + z_m e_m + \sum_{i=m+1}^{d} \psi_j(z_2,\ldots, z_m) e_j ; (z_1, \ldots, z_m) \in G \right\rbrace
\end{equation}
where $\left\lbrace e_j \right\rbrace _{j=2}^{m}$ is an orthonormal basis of the tangent plane to $\mbox{supp}(\lambda_{0})$ at $y$ and $\left\lbrace e_j \right\rbrace _{j=1}^{d}$ is a completion of $\left\lbrace e_j \right\rbrace _{j=1}^{m}$ to an orthonormal basis of $\RR^{d}$.
 We claim that $\Sigma$ is a $C^{1}$-manifold in the neighborhood $U=\left\lbrace se_1+ z'; (s,z) \in (t-1, t+1) \times U'\right\rbrace$ of $te_1 + y$. Indeed, if $z \in \Sigma \cap U$, then by $\eqref{transinv}$ and $\eqref{m-1-manif}$ we can write 
\begin{equation}\label{m-manifold}
 z= z_1 e_1 + z_2 e_2 + \ldots + z_m e_m + \sum_{i=m+1}^{d} \psi_j(z_2,\ldots, z_m)e_j, 
\end{equation}
where $z_j = \left\langle z,e_j \right\rangle$ for $j>1$ and $z_1 \in (r-1,r+1)$.

We go back to the proof of $\eqref{almostproduct}$. Suppose that $\eta \in \mathcal{S}_{\lambda}$. Then in particular, $\eta \in \Sigma$ and hence $\eta= te_1+y$ where $t \in \RR$, $y \in A$. If $y$ were a flat point of $\lambda_0$, then $\eta$ would be a flat point of $\lambda$. Therefore, $\eta \in \RR \times \mathcal{S}_{\lambda_0}$.

We deduce from $\eqref{almostproduct}$ that 
\begin{equation} \label{hausdorffdimproduct}
dim_{\mathcal{H}}(\mathcal{S}_{\lambda}) \leq dim_{\mathcal{H}}(S_{\lambda_0}) + 1.
\end{equation} 
Note that this inequality holds because $\mathcal{S}_{\lambda}$ and $S_{\lambda_0}$ are Borel sets and the packing dimension of a line is the same as its Hausdorff dimension.
But since $\mathcal{H}^{s}(\mathcal{S}_{\lambda})>0$, 
\begin{equation} \label{hausdorffdimproduct2}
dim_{\mathcal{H}}(\mathcal{S}_{\lambda}) \geq s. 
\end{equation} 
Combining $\eqref{hausdorffdimproduct}$ and $\eqref{hausdorffdimproduct2}$, we get:
\begin{equation}
s-1 \leq dim_{\mathcal{H}}(S_{\lambda_0}).
\end{equation}
On the other hand, $S_{\lambda_0}$ being the singular set of an $(m-1)$-uniform measure, the induction hypothesis implies that $dim_{\mathcal{H}}(S_{\lambda_0}) \leq m-4$. Therefore $s \leq m-3$.

We have proven that 
$$ \mathcal{H}^{s}(\mathcal{S}_{\mu})>0 \implies s \leq m-3.$$
Therefore,
$$dim_{\mathcal{H}}(\mathcal{S}_{\mu}) \leq m-3.$$

\end{proof}
  \section*{Acknowledgements}
 I would like to thank my graduate advisor T. Toro for her invaluable help. I would also like to thank Max Engelstein for helpful discussions during his visit to the University of Washington.

\end{document}